\documentclass[12pt]{amsart}
\usepackage{amsmath}
\usepackage{amssymb}
\usepackage{amsthm}
\usepackage[all]{xy}
\usepackage{graphicx}
\usepackage[margin=1in]{geometry}
\usepackage{color}
\usepackage{hyperref}
\usepackage{verbatim}
\usepackage{pb-diagram,pb-xy}

\newcommand{\R}{\ensuremath{\mathbb{R}}}
\newcommand{\N}{\ensuremath{\mathbb{N}}}

\newtheorem{theorem}{Theorem}[section]
\newtheorem{lemma}[theorem]{Lemma}
\newtheorem{proposition}[theorem]{Proposition}

\newtheorem{corollary}[theorem]{Corollary}

\theoremstyle{definition} 
\newtheorem{defn}{Definition}[section]
\newtheorem{Construction}{Construction}[section]
 
\newtheorem{remark}{Remark}[section] 
\newtheorem{Notation}{Notational Convention}[section]

\newcommand{\Th}{{\mathsf T}\!{\mathsf h}}
\newcommand{\MT}{{\mathsf M}{\mathsf T}}
\newcommand{\Mor}{\text{mor}}
\newcommand{\Spec}{\textbf{Spec}}
\newcommand{\Top}{\textbf{Top}}
\newcommand{\Ob}{\text{ob}}

\newcommand{\Grass}[3]{G(\Sigma_{k},{#2} , {#3}, {#1})}

\newcommand{\kset}[1]{\langle {#1} \rangle}
\newcommand{\kseto}[1]{\widehat{\langle {#1} \rangle}}
\newcommand{\Sig}[2]{\Sigma_{{#1}}^{{#2}}}
\newcommand{\Cob}{{\mathbf{Cob}}}
\DeclareMathOperator*{\tCofibre}{tCofibre}
\DeclareMathOperator*{\Cofibre}{Cofibre}

\DeclareMathOperator*{\colim}{colim}

\DeclareMathOperator*{\Diff}{Diff}

\DeclareMathOperator*{\BDiff}{BDiff}

\DeclareMathOperator*{\Maps}{Maps}

\DeclareMathOperator*{\Fib}{Fib}
\DeclareMathOperator*{\op}{op}

\author{Nathan Perlmutter}

\address{Department of Mathematics, University of Oregon, Eugene, OR,
  97403, USA} 
  
  \email{nperlmut@uoregon.edu}

\title[Cobordism categories of manifolds with
singularities]{Cobordism categories of manifolds \\ with Baas-Sullivan
  singularities, Part 2}
  
\begin{document}
\maketitle
  \begin{abstract} For a given list of closed manifolds $\Sigma_k=(P_1,\ldots,P_k)$, we construct a cobordism category $\mathbf{Cob}_{d}^{\Sigma_{k}}$ of manifolds with Baas-Sullivan singularities of type $\Sigma_k$. Our main results identify the homotopy type of the classifying space $B\mathbf{Cob}_{d}^{\Sigma_{k}}$, with that of the infinite loop-space of a certain spectrum, $\mathsf{MT}_{\Sigma_{k}}(d)$. Our results generalize of the work of Galatius, Madsen, Tillmann, and Weiss from \cite{GMTW 09}. It also generalizes the work of Genauer on the cobordism categories of manifolds with corners from \cite{G 08}.
    \end{abstract}

\section{Introduction and Statement of Main Results} 
\label{section: Introduction}

\subsection{Background and motivation} 
In \cite{GMTW 09}, Galatius, Madsen, Tillmann, and Weiss construct a
cobordism category $\mathbf{Cob}_{d}$ and identify the homotopy type
of its classifying space with that of the infinite loop-space
$\Omega^{\infty -1}\MT(d)$, where $\MT(d)$ is the Thom spectrum
associated to the inverse to the universal, $d$-dimensional vector
bundle $U_{d} \rightarrow BO(d)$. This fundamental result lead to a new
proof of the Mumford conjecture (given in \cite{GMTW 09}) and laid the foundations for the study of the moduli
spaces of smooth manifolds, see \cite{GRW 14}.

In \cite{P 13} we generalized the main result from \cite{GMTW 09} to the
case of manifolds with a single Baas-Sullivan singularity. 
In this paper we determine the homotopy type of a cobordism
category of manifolds with multiple Baas-Sullivan singularities. 
This work simultaneously generalizes the results from \cite{GMTW 09} and the work 
on cobordism categories of manifolds with corners from \cite{G 08}.

\subsection{Manifolds with Baas-Sullivan singularities}\label{motiv}
We fix a sequence $\Sigma=(P_{1}, \dots, P_{k}, \dots )$ of smooth,
closed manifolds and let $p_{i}$ denote the dimension of the manifold
$P_{i}$ for each $i \in \N$.  Then for $k \in \N$, we let $\Sigma_{k}$
denote the truncated list $(P_{1}, \dots, P_{k})$.  Following \cite{Ba
  73} and \cite{B 92}, a manifold with \textit{Baas-Sullivan}
singularities modeled on $\Sigma_{k}$ is a smooth manifold $W$
equipped with the following extra structure:
\begin{itemize} \itemsep.2cm
\item The boundary $\partial W$ is given a decomposition
$
\partial W \; = \; \partial_{0}W \cup \partial_{1}W \cup \cdots \cup
\partial_{k}W
$
with the property that for each subset $I \subseteq \{0, 1, \dots,
k\}$, the intersection
$\partial_{I}W \; := \; \bigcap_{i\in I}\partial_{i}W$ is a $(\dim(W)
-|I|)$-dimensional manifold, where $|I|$ is the order of
$I$.  
\item
For each subset $I \subseteq \{0, 1,
\dots, k\}$, the manifold $\partial_{I}W$ is equipped with a factorization,
$$\partial_{I}W \; =  \; \beta_{I}W\times P^{I},$$
where $P^{I} = \prod_{i\in I}P_{i}$ and $\beta_{I}W$ is a manifold of dimension, $\dim(W) - |I| - \sum_{i \in I}{p_{i}}$. 
\end{itemize}

The submanifold $\partial_{0}W$ is referred to as the \textit{boundary} of
the $\Sigma_{k}$-manifold $W$. 
A compact $\Sigma_{k}$-manifold $W$ is said to be \textit{closed} if $\partial_{0}W = \emptyset$.
Two closed $\Sigma_{k}$-manifolds $M_{1}$ and $M_{2}$ are said to be $\Sigma_{k}$-cobordant if there exists a compact $\Sigma_{k}$-manifold $W$ such that $\partial_{0}W = M_{1}\sqcup M_{2}$. 

We denote by $\Omega_{*}$ the graded cobordism group of unoriented
manifolds, and by $\Omega^{\Sigma_{k}}_{*}$ the graded cobordism group of
manifolds with \emph{Baas-Sullivan singularities of type}
$\Sigma_{k}$. We denote, $\Omega^{\Sigma_{0}}_{*} :=
\Omega_{*}$.
For varying
integer $k$, the groups $\Omega^{\Sigma_k}_{*}$ are related to each other
by the well-known \textit{Bockstein-Sullivan} exact couple:
\begin{equation} \label{eq: exact couple}
\xymatrix{
\Omega^{\Sigma_{k-1}}_{*} \ar[drr]_{i_k} &&&& 
\Omega^{\Sigma_{k-1}}_{*} \ar[llll]_{\times P_{k}}\\
&& \Omega^{\Sigma_{k}}_{*} \ar[urr]_{\beta_{k}} &&}
\end{equation}
The map $\times P_{k}$ is given by multiplication by the manifold
$P_{k}$ and is of degree $p_{k}$, $i_k$ is the map given by inclusion, and $\beta_{k}$ is the
degree $-(p_{k}+1)$ map given by sending a $\Sigma_{k}$-manifold
$W$ to the $\Sigma_{k-1}$-manifold $\beta_{k}W$.

We will construct a cobordism category of manifolds with Baas-Sullivan
singularities $\mathbf{Cob}_{d}^{\Sigma_{k}}$ (see also \cite{Ba 09}), and then proceed to determine
the homotopy-type of its classifying space, $B\mathbf{Cob}_{d}^{\Sigma_{k}}$. 
Roughly, the objects of $\mathbf{Cob}_{d}^{\Sigma_{k}}$ are given by embedded, closed, $\Sigma_{k}$-manifolds $M \subset \R^{\infty}$ of dimension $d-1$, and morphisms are given by embedded $\Sigma_{k}$-cobordisms 
$W \subset [0, 1]\times\R^{\infty}$ of dimension $d$, such that such that $\partial W  = W\cap(\{0, 1\}\times\R^{\infty})$ is an orthogonal intersection.
A more precise definition of the category $\mathbf{Cob}_{d}^{\Sigma_{k}}$ is given in Section \ref{section: the cobordism category} using important preliminary results about the differential topology of $\Sigma_{k}$-manifolds developed in the sections leading up to it. 
Following \cite{GMTW 09}, the category $\mathbf{Cob}_{d}^{\Sigma_{k}}$ is topologized in such a way so that there are homotopy equivalences,
 $$\xymatrix{
 \Ob(\mathbf{Cob}_{d}^{\Sigma_{k}}) \; \simeq \;
{\displaystyle \coprod_{[M]}}\BDiff_{^{\Sigma_{k}}}(M), \quad \quad
 \Mor(\mathbf{Cob}_{d}^{\Sigma_{k}}) \; \simeq \; \Ob(\mathbf{Cob}_{d}^{\Sigma_{k}})\amalg
{\displaystyle \coprod_{[W]}}\BDiff_{\Sigma_{k}}(W)
 },$$
  where $M$ varies over
 diffeomorphism classes of $(d-1)$-dimensional closed $\Sigma_{k}$-manifolds and $W$ varies over
 diffeomorphism classes of $d$-dimensional $\Sigma_{k}$-cobordisms. 
 The group $\Diff_{\Sigma_{k}}(W)$ is
 defined to be the group of all diffeomorphisms $g: W \longrightarrow W$ such that 
 $g(\partial_{I}W) = \partial_{I}W$ for all  subsets $I \subseteq \{0, \dots, k\}$, and with the
 additional property that
 the restriction of $g$ to $\partial_{I}W$ has the factorization
$g\mid_{\partial_{I}W} \; = \; g_{\beta_{I}}\times Id_{P^{I}}$, 
 where $g_{\beta_{I}}: \beta_{I}W \rightarrow \beta_{I}W$ is a
 diffeomorphism. 

 This construction is a generalization of the category $\mathbf{Cob}_{d}$ from \cite{GMTW 09} in the sense that $ \mathbf{Cob}^{\Sigma_{k}}_{d}$ is isomorphic (as a topological category) to $\mathbf{Cob}_{d}$ in the case that $\Sigma_{k} = (\emptyset, \dots, \emptyset)$. 
 Furthermore, in the case that $P_{i}$ is the single point space $\star$ for all $i \leq k$, then the category $\mathbf{Cob}^{\Sigma_{k}}_{d}$ coincides with the cobordism category of \textit{manifolds with corners} studied by Genauer in \cite{G 08}.
 
\subsection{First results}
We will need to compare the categories $\mathbf{Cob}_{d}^{\Sigma_{k}}$
to other similar cobordism categories consisting of manifolds with different singularity types. 
For a non-negative integer $\ell \leq k$, 
we denote by $\Sigma^{\ell}_{k}$ the list obtained from $\Sigma_{k}$ by replacing the manifold $P_{i}$ with a the single point space $\star$ if $i > \ell$, i.e.\ $\Sigma^{\ell}_{k} = (P_{1}, \dots, P_{\ell}, \star, \dots, \star)$. 
 Using these singularity lists, the categories
$\mathbf{Cob}_{d}^{\Sigma^{\ell}_{k}}$ are defined in the same way as before. 
The category $\mathbf{Cob}_{d}^{\Sigma^{0}_{k}}$ is the cobordism category of \textit{manifolds with corners} studied in \cite{G 08}. 
The categories $\mathbf{Cob}_{d}^{\Sigma^{\ell}_{k}}$ serve as intermediates between the category of manifolds with corners and the cobordism category of $\Sigma_{k}$-manifolds $\mathbf{Cob}_{d}^{\Sigma_{k}}$, which is our main object of interest.
For any pair of integers $\ell < k$ we have a commutative diagram,
\begin{equation} \label{equation: pull-back square}
\xymatrix{ 
\mathbf{Cob}_{d}^{\Sigma^{\ell+1}_{k}} \ar[d]^{\beta_{\ell+1}} \ar[rrr] &&& \mathbf{Cob}_{d}^{\Sigma^{\ell}_{k}}
  \ar[d]^{\partial_{\ell + 1}}\\ 
  \mathbf{Cob}_{d-p_{\ell+1}-1}^{\Sigma^{\ell}_{k-1}}
  \ar[rrr]^{\tau_{P_{\ell +1}}} &&&
  \mathbf{Cob}_{d-1}^{\Sigma^{\ell}_{k-1}}
  }
\end{equation}
where $\partial_{\ell+1}$
is the functor defined by sending a cobordism $W$ (which is a morphism) to
$\partial_{\ell+1}W$, $\beta_{\ell+1}$ sends a cobordism $W$ to 
$\beta_{\ell+1}W$,  the
map $\tau_{P_{\ell+1}}$ is the functor defined by sending a cobordism
$W$ to the product $W\times P_{\ell+1}$, and the top-horizontal map is the functor defined by considering a $\Sigma_{k}^{\ell+1}$-manifold a $\Sigma_{k}^{\ell}$-manifold. 
It follows from the construction of the categories $\mathbf{Cob}_{d}^{\Sigma^{\ell+1}_{k}}$ in Section \ref{section: the cobordism category} (and the preliminary results developed in the sections leading up to this), that the above diagram is a pull-back square. 
The fact that (\ref{equation: pull-back square}) is a pull-back square for all $\ell < k$ implies that the category $\mathbf{Cob}_{d}^{\Sigma^{\ell+1}_{k}}$ can be defined inductively as the iterated pull-back of the maps $\tau_{P_{\ell+1}}$ and $\partial_{\ell+1}$, starting with $\ell = 0$. 

Passing to classifying spaces, we obtain the commutative square,
\begin{equation} \label{eq: cartesian square intro} 
\xymatrix{
B\mathbf{Cob}_{d}^{\Sigma^{\ell+1}_{k}} \ar[rrr] \ar[d]^{B(\beta_{\ell+1})} &&& B\mathbf{Cob}_{d}^{\Sigma^{\ell}_{k}} \ar[d]^{B(\partial_{\ell + 1})}\\
B\mathbf{Cob}_{d-p_{\ell+1}-1}^{\Sigma^{\ell}_{k-1}} \ar[rrr]^{B(\tau_{P_{\ell+1}})} &&& B\mathbf{Cob}_{d-1}^{\Sigma^{\ell}_{k-1}}.
}
\end{equation}
This brings us to the statement of our first result, which is proven in Section \ref{section: A fibre sequence of classifying spaces}. 
\begin{theorem}  \label{thm: homotopy cartesian} The diagram {\rm (\ref{eq: cartesian square intro})} is homotopy-cartesian for all $\ell$ and $k$ with $0 \leq \ell < k$. 
\end{theorem}
In Proposition \ref{proposition: homotopy invariance of times P}, we show that the homotopy class of the bottom-horizontal map $B(\tau_{P_{\ell + 1}})$ depends only the cobordism class of the closed manifold $P_{\ell + 1}$. 
Applying this fact to the homotopy cartesian diagram (\ref{eq: cartesian square intro}) yields the following corollary:
 \begin{corollary} \label{cor: cobordism independence}
Let $\widehat{\Sigma}$ be a sequence of manifolds obtained by
replacing each manifold $P_{i}$ from the sequence $\Sigma$, by a manifold
$\widehat{P}_{i}$ that is cobordant to $P_{i}$. 
Then for all $0 \leq \ell < k$, the classifying space
$B\mathbf{Cob}^{\widehat{\Sigma}^{\ell}_{k}}_{d}$ is weakly homotopy
equivalent to $B\mathbf{Cob}^{\Sigma^{\ell}_{k}}_{d}$.
\end{corollary}
We then identify the homotopy fibre of the vertical
maps of the above homotopy-cartesian square with the classifying space
$B\mathbf{Cob}_{d}^{\Sigma^{\ell-1}_{k-1}}$. 
\begin{corollary} \label{thm: homotopy fibre sequence classifying space} For all $k \in \N$, there is a homotopy fibre-sequence,
$$\xymatrix{
B\mathbf{Cob}_{d}^{\Sigma^{\ell}_{k-1}} \ar[rr] &&
    B\mathbf{Cob}_{d}^{\Sigma^{\ell+1}_{k}} \ar[rr]^{B(\beta_{\ell+1})} &&
    B\mathbf{Cob}_{d-1-p_{k}}^{\Sigma^{\ell}_{k-1}}.
    }$$
\end{corollary}

\subsection{The main result}
To state our main theorem, we will need some more terminology. 
For an integer $k \in \N$, let $\langle k \rangle$ denote the set $\{1, \dots, k\}$ and then let $2^{\langle k \rangle}$ denote the category whose objects are given by the subsets of $\langle k \rangle$ and whose morphisms are given by the inclusion maps. 
We will need to consider functors $(X_{\bullet})^{\text{op}}: 2^{\langle k \rangle} \longrightarrow \Spec$, where $\Spec$ denotes the category of spectra. 
Such a functor $X_{\bullet}$ will be referred to as a \textit{$k$-cubic spectrum}. 
For any $k$-cubic spectrum $X_{\bullet}$, we denote by $t\Cofibre(X_{\bullet})$ the \textit{total homotopy cofibre} of the functor $X_{\bullet}$ (see Section \ref{subsection: total cofibres} for the definition, or see also \cite[Definition 3.5]{M 10}).

In Section \ref{section: k spaces} we construct a $k$-cubic spectrum $\MT_{\Sigma_{k}}(d)_{\bullet}$ such that for each subset $I \subseteq \langle k \rangle$, the spectrum $\MT_{\Sigma_{k}}(d)_{I}$ is equal to the spectrum 
$\Sigma^{-|I|}\MT(d-|I|-p_{I})$ from \cite{GMTW 09}, where $p_{I} = \sum_{i \in I}p_{i}$. 
We then let $\MT_{\Sigma_{k}}(d)$ denote the spectrum given by the total homotopy cofibre, $t\Cofibre(\MT_{\Sigma_{k}}(d)_{\bullet})$. 
The $k$-cubic spectrum $\MT_{\Sigma_{k}}(d)_{\bullet}$ is constructed in such a way so that in the case that $\Sigma_{k}$ is a list of single points, the spectrum $\MT_{\Sigma_{k}}(d)$ reduces to the spectrum considered in \cite{G 08}. 
In Section \ref{section: Main Theorem}, we prove the following theorem. 
\begin{theorem} \label{thm: weak homotopy equivalence} There is a weak homotopy equivalence,
$B\Cob^{\Sigma_{k}}_{d} \simeq \Omega^{\infty-1}\MT_{\Sigma_{k}}(d)$. 
\end{theorem}

\begin{remark} To simplify the exposition and notation we will only treat unoriented manifolds Baas-Sullivan singularities. 
However, all of our constructions could be carried through for manifolds equipped with tangential structures and analogous results could be obtained with little extra work. 
We have not done these things so as to avoid a notational nightmare. 
\end{remark}

\subsection{Organization}
In Section \ref{section: The Category 1} we give a rigorous definition of $\Sigma_{k}$-manifolds. 
In Section \ref{section: mapping spaces} we define certain important mapping spaces associated to $\Sigma_{k}$-manifolds. 
In Section \ref{section: the cobordism category} we define the cobordism category $\mathbf{Cob}^{\Sigma_{k}}_{d}$ of $\Sigma_{k}$-manifolds. 
In Section \ref{section: sheaf model} we construct certain sheaves modeled on the cobordism category $\mathbf{Cob}^{\Sigma_{k}}_{d}$ and in Section \ref{section: A fibre sequence of classifying spaces} we prove Theorem \ref{thm: homotopy cartesian}, using all of the constructions developed in the previous sections. 
In Section \ref{section: k spaces} we construct the spectrum that appears in the statement of Theorem \ref{thm: weak homotopy equivalence} and in Section \ref{section: Main Theorem} we prove Theorem \ref{thm: weak homotopy equivalence}. 
In the appendix we prove a technical lemma which is used in Section \ref{eq: P-embeddings} to prove Theorem \ref{thm: contractibility}. 

\subsection{Acknowledgements}
The author would like to thank Boris Botvinnik for suggesting this
particular problem and for numerous helpful discussions on the subject
of this paper. 
The author is also grateful to Nils Baas and Marius Thaule for some very helpful conversations and remarks
and to Oscar Randal-Williams for
critical comments on the first part of this project.  

 \section{Manifolds with Baas-Sullivan singularities} \label{section: The Category 1}
Fix once and for all a sequence of closed manifolds $\Sigma := (P_{1}, \dots, P_{j}, \dots )$. 
For each $i \in \N$, we let $p_{i}$ denote the integer $\dim(P_{i})$. 
For $k \in \N$, let $\Sigma_{k}$, denote the truncated list $(P_{1}, \dots, P_{k})$ comprised of the first $k$-elements of $\Sigma$ ($\Sigma_{0}$ is then defined to be the empty list). 
If $\ell$ is a non-negative integer less than or equal to $k$, we let $\Sigma^{\ell}_{k}$ denote the list obtained from $\Sigma_{k}$ by replacing the manifold $P_{i}$ by $\star$ (the single point space) if $i > \ell$. 
\begin{remark}
The main objects of interest are $\Sigma_{k}$-manifolds. 
However, for technical reasons it will be important to also study manifolds with singularities modeled on the lists $\Sigma^{\ell}_{k}$ for $\ell \leq k$. 
The $\Sigma^{\ell}_{k}$-manifolds will serve as intermediates between \textit{manifolds with corners} and $\Sigma_{k}$-manifolds and will be useful in many of our constructions, see (\ref{equation: pull-back square}) and Theorem \ref{thm: homotopy cartesian}.
\end{remark}
\begin{Notation} The following are some notational conventions that we use in this section and throughout the paper.
\begin{itemize} \itemsep.2cm
\item
We let $\langle k \rangle$ denote the set $\{1, \dots, k\}$ and let $\widehat{\langle k \rangle}$ denote the set $\{0, 1, \dots, k\}$. 
\item
If $\bar{\epsilon} = (\epsilon_{0}, \dots, \epsilon_{k})$ is a list of positive real numbers, we let $[0, \bar{\epsilon})^{k+1}$ denote the space given by 
$\{(x_{0}, \dots, x_{k}) \in [0, \infty)^{k+1} \; | \; x_{i} < \epsilon_{i} \; \}.$
\item
If $I \subseteq \widehat{\langle k \rangle}$ is a subset, we let $[0, \bar{\epsilon})^{k+1}_{I}$ denote the subspace of $[0, \bar{\epsilon})^{k+1}$ given by 
$\{(x_{0}, \dots, x_{k}) \in [0, \bar{\epsilon})^{k+1} \; | \; \text{$x_{i} = 0$ when $i \notin I$}\; \}.$ 
Similarly, we let $[0, \infty)^{k+1}_{I}$ denote the subspace of $[0, \infty)^{k+1}$ given by 
$\{(x_{0}, \dots, x_{k}) \in [0, \infty)^{k+1} \; | \; \text{$x_{i} = 0$ when $i \notin I$}\; \}.$ 
\vspace{.1cm}
\end{itemize}
\end{Notation}
We now give a rigorous definition of $\Sigma^{\ell}_{k}$-manifold. 
By this definition, $\Sigma^{\ell}_{k}$-manifolds are a generalization of \textit{manifolds with corners}. For background on manifolds with corners see \cite{G 08} and \cite{L 00}.
\begin{defn} \label{defn: closed Manifold def} 
Let $\ell \leq k$ be non-negative integers. 
A smooth, $d$-dimensional manifold $M$ is said to be a $\Sigma^{\ell}_{k}$-manifold if it is equipped with the following extra structure:
\begin{enumerate} \itemsep.3cm
\item[i.] The boundary $\partial M$ is given a decomposition
$\partial M \; = \; \partial_{0}M \cup \cdots \cup \partial_{k}M$
 into a union of $(d-1)$-dimensional manifolds such that for all subsets $I
  \subseteq \kseto{k}$, the intersection
$$\partial_{I}M \; := \; \bigcap_{i \in I}\partial_{i}M$$ 
is a $(d- |I|)$-dimensional manifold, with boundary given by
$$ 
\partial(\partial_{I}M) \; = \; \bigcup_{j \in \langle k
  \rangle}(\partial_{j}M\cap \partial_{I}M).
$$
\item[ii.] For each subset $I \subseteq \kseto{k}$,
there are embeddings
$ h_{I}: \partial_{I}M\times [0, \infty)^{k+1}_{I} 
  \longrightarrow \; M 
$
that satisfy the following two compatibility conditions:
\begin{enumerate} \itemsep.2cm
\item[(a)] For each pair of subsets $J \subseteq I \subseteq \widehat{\langle k \rangle}$, 
the embedding $h_{J}$
    maps the subspace 
    $$\partial_{J}M\times[0,1)^{k+1}_{
        J\setminus I} \; \subset \;
      \partial_{J}M\times[0,1)^{k+1}_{J }
$$ 
into $\partial_{I}M\times [0,1)^{k+1}_{I}$. 
\item[(b)] For each pair of subsets $I \subseteq J \subseteq \widehat{\langle k \rangle}$, the following diagram commutes, 
$$
\xymatrix{
\partial_{J}M\times[0,1)^{k+1}_{ J } \ar[d]^{=}
\ar[rrrrrd]^{h_J} &&& &&
\\
\partial_{J}M\times[0,1)^{k+1}_{ J\setminus I }\times[0,1)^{k+1}_{I}  \ar[rrr]_{h_{J\setminus I}\times Id} &&& \partial_{I}M\times[0,1)^{k+1}_{I } \ar[rr]_{h_{I}} && {^{\mbox{\ }}M},}
$$ 
where the left vertical map is the standard identification
$$\partial_{J}M\times[0,1)^{k+1}_{
          J\setminus I}\times[0,1)^{k+1}_{
            I} \; = \; \partial_{J}M\times[0,1)^{k+1}_{J}.$$
          \end{enumerate}
\item[iii.] For each subset $I \subseteq \langle \ell \rangle$ there are
  diffeomorphisms,
$$\xymatrix{\psi_{I}: \partial_{I}M \ar[rr]^{\cong} &&
    \beta_{I}M\times P^{I}}$$ where $\beta_{I}M$ is a $(d- |I| -
  p_{I})$-dimensional manifold and $P^{I} = \prod_{i \in I}P_{i}$. 
  These diffeomorphisms are subject to the following compatibility condition:
  if $I \subseteq J \subseteq \kset{\ell}$ are subsets and
  $\iota_{J,I}: \partial_{J}M \hookrightarrow \partial_{I}M$ is the
  corresponding inclusion, then the map
$$\psi_{I}\circ\iota_{J,I}\circ\psi^{-1}_{J}: \beta_{J}M\times P^{J}
\longrightarrow \beta_{I}M\times P^{I}$$ is the identity on the direct
factor of $P^{I}$ in the product $P^{J} = \prod_{j \in J}P_{j}$.
\end{enumerate}
\end{defn}
We refer to the embeddings $h_{I}$ from ii. as \textit{collars} and the diffeomorphisms from iii. as the \textit{structure maps}.
It follows directly from the above definition that for a subset $I \subseteq \kseto{k}$, the manifolds
$\partial_{I}M$ and $\beta_{I}M$ are $\Sigma^{\ell}_{k}$-manifolds of dimension $d - |I|$ and $d - |I| - \sum_{i\in I}p_{i}$ respectively. 
Furthermore, both $\partial_{\ell+1}M$ and $\beta_{\ell+1}M$ have the structure of $\Sigma^{\ell}_{k-1}$-manifolds. 
The $\Sigma^{\ell}_{k}$-manifold $\partial_{0}M$ is said to be the \textit{boundary} of $M$.  
If $M$ is compact and $\partial_{0}M = \emptyset$, then $M$ is said to be a \textit{closed} $\Sigma^{\ell}_{k}$-manifold. 

We will need to consider maps between $\Sigma^{\ell}_{k}$-manifolds and arbitrary topological spaces. 
\begin{defn} \label{defn: sigma-k map}
Let $M$ be a $\Sigma^{\ell}_{k}$-manifold and let $X$ be a topological space. 
A map 
$$f: M \longrightarrow X$$
is said to be a $\Sigma^{\ell}_{k}$-map if for each subset $I \subseteq \langle \ell \rangle$, there is a map $f_{\beta_{I}}: \beta_{I}M \longrightarrow X$ such that the restriction map $f|_{\partial_{I}M}: \partial_{I}M \longrightarrow X$ has the factorization,
$$\xymatrix{
\partial_{I}M \ar[rr]^{\phi_{I}} && \beta_{I}M\times P^{I} \ar[rrr]^{\text{proj}_{\beta_{I}M}} &&& \beta_{I}M \ar[rr]^{f_{\beta_{I}}} && X .
}$$
If $X$ is a smooth manifold, then $f$ is called a \textit{smooth $\Sigma^{\ell}_{k}$-map} if $f$ is smooth when considering $M$ as a smooth manifold with corners.
\end{defn}

We will ultimately need to consider submersions of $\Sigma^{\ell}_{k}$-manifolds. 
\begin{defn} \label{defn: sigma-k submersion}
A smooth $\Sigma^{\ell}_{k}$-map $f: M \longrightarrow X$ is said to be a $\Sigma^{\ell}_{k}$-submersion if $f$ is a submersion (when considering $M$ as a smooth manifold with corners) and $f|_{\partial_{I}M}: \partial_{I}M \longrightarrow X$ is a submersion for every subset $I \subseteq \widehat{\langle k \rangle}$. 
\end{defn}
Transversality will also play an important role. 
\begin{defn} \label{defn: transversality}
Let $X$ be a smooth manifold and let $K \subset X$ be a submanifold. 
A smooth $\Sigma^{\ell}_{k}$-map $f: M \longrightarrow X$ is said to be \textit{$\Sigma^{\ell}_{k}$-transverse} to $K$ if $f$ is transverse to $K$ and if for each subset $I \subseteq \widehat{\langle k \rangle}$, the map $f|_{\partial_{I}M}: \partial_{I}M \longrightarrow X$ is transverse to $K$.
\end{defn}
It follows easily that if the $\Sigma^{\ell}_{k}$-map $f: M \longrightarrow X$ is transverse to $K$ then the pre-image $f^{-1}(K)$ has the structure of a $\Sigma^{\ell}_{k}$-manifold. 

We are ultimately interested in the cobordism theory of $\Sigma^{\ell}_{k}$-manifolds. 
\begin{defn} \label{defn: cobordism}
Let $M_{a}$ and $M_{b}$ be closed $\Sigma^{\ell}_{k}$-manifolds of dimension $d$ and $W$ be a compact $\Sigma^{\ell}_{k}$-manifold of dimension $d+1$. 
If $\partial_{0}W = M_{a}\sqcup M_{b}$, then the triple $(W, M_{a}, M_{b})$ is said to be a \textit{$\Sigma^{\ell}_{k}$-cobordism triple}. 
Two closed $\Sigma^{\ell}_{k}$-manifolds $M_{a}$ and  $M_{b}$ are said to be $\Sigma^{\ell}_{k}$-cobordant if there exists a compact $(d+1)$-dimensional $\Sigma^{\ell}_{k}$-manifold $W$ such that $\partial_{0}W = M_{a}\sqcup M_{b}$. 
\end{defn}
\begin{remark} \label{remark: bordism group}
Using the above definition, one can construct the cobordism groups of $\Sigma^{\ell}_{k}$-manifolds, denoted by $\Omega^{\Sigma^{\ell}_{k}}_{*}$.
We refer the reader to \cite{B 92} for details on the construction. 
\end{remark}

\section{Mapping Spaces} \label{section: mapping spaces}
\subsection{Diffeomorphisms}
For what follows, let $M_{a}$ and $M_{b}$ be $\Sigma^{\ell}_{k}$-manifolds. 
For each subset $I \subseteq \kseto{k}$, we denote by $h^{a}_{I}$ and $h^{b}_{I}$ the collar embeddings associated to $M_{a}$ and $M_{b}$ and denote by
 $\psi^{a}_{I}$ and $\psi^{b}_{I}$ the \textit{structure maps} from, condition iii. of Definition \ref{defn: closed Manifold def}. 
\begin{defn} \label{defn: P-morphism}
A smooth map $f: M_{a} \longrightarrow M_{b}$ is said to be a \textit{$\Sigma^{\ell}_{k}$-morphism} if the following conditions are satisfied:
\begin{enumerate} \itemsep.2cm
\item[i.] For all subsets $I \subseteq \kseto{k}$, we have $f(\partial_{I}M_{a}) \subset \partial_{I}M_{b}$.
\item[ii.] The map $f$ interacts with the collars in the following way.
There exists a list of positive integers $\bar{\epsilon} = (\epsilon_{0}, \dots, \epsilon_{k})$
such that for each subset $I \subseteq \kseto{k}$, 
$$ 
f\circ h_{I}(x, t_{1}, \dots, t_{k}) \; = \;
  h_{I}(f|_{\partial_{I}M_{a}}(x),  t_{0}, \dots, t_{k})
$$ 
where $x \in \partial_{I}M_{a}$ and $(t_{0}, \dots, t_{k}) \in [0,\bar{\epsilon})^{k+1}_{I}$.
\item[iii.] For each subset $I \subseteq \kset{\ell}$, the
  restrictions $f|_{\partial_{I}M}$ have the factorizations,
$$f\mid_{\partial_{I}M_{a}} = (\psi^{a})_{I}^{-1}\circ(f_{\beta_{I}M_{a}}\times Id_{P^{I}})\circ\psi^{b}_{I}$$
where $f_{\beta_{I}M_{a}}: \beta_{I}M_{a} \longrightarrow \beta_{I}M_{b}$ is a smooth map satisfying conditions i. and ii. given above.
\end{enumerate}

\end{defn}
We let $C^{\infty}_{\Sigma^{\ell}_{k}}(M_{a}, M_{b})$ denote the space of $\Sigma^{\ell}_{k}$-morphisms $M_{a} \rightarrow M_{b}$, topologized as a subspace of the space of smooth maps $M_{a} \rightarrow M_{b}$ in the $C^{\infty}$-topology.
For a $\Sigma^{\ell}_{k}$-manifold $M$, we let $C^{\infty}_{\Sigma^{\ell}_{k}}(M)$ denote the space $C^{\infty}_{\Sigma^{\ell}_{k}}(M, M)$. 
We will need to consider diffeomorphisms of $\Sigma^{\ell}_{k}$-manifolds as well. 
\begin{defn} A smooth map between $\Sigma^{\ell}_{k}$-manifolds $f: M_{a} \longrightarrow M_{b}$, is said to be a \textit{$\Sigma^{\ell}_{k}$-diffeomorphism} if it is both a diffeomorphism as a map of smooth manifolds and a $\Sigma^{\ell}_{k}$-morphism, i.e. it satisfies all conditions of Definition \ref{defn: P-morphism}. 
\end{defn}
 We denote by $\Diff_{\Sigma^{\ell}_{k}}(M_{a}, M_{b})$ the space of $\Sigma^{\ell}_{k}$-diffeomorphisms from $M_{a}$ to $M_{b}$, where
 $\Diff_{\Sigma^{\ell}_{k}}(M_{a}, M_{b})$ is topologized as a subspace of $C^{\infty}_{\Sigma^{\ell}_{k}}(M_{a}, M_{b})$. 
 For a $\Sigma^{\ell}_{k}$-manifold $M$, we let $\Diff_{\Sigma^{\ell}_{k}}(M)$ denote the space $\Diff_{\Sigma^{\ell}_{k}}(M, M)$.
 The space $\Diff_{\Sigma^{\ell}_{k}}(M)$ has the structure of a topological group with product given by composition.

The next proposition is proven in the same way as \cite[Proposition 3.1]{P 13}.
\begin{proposition} \label{prop: P-diff open set}
For any two compact $\Sigma^{\ell}_{k}$-manifolds $M_{a}$ and $M_{b}$, $\Diff_{\Sigma^{\ell}_{k}}(M_{a}, M_{b})$ is an open subset of $C^{\infty}_{\Sigma^{\ell}_{k}}(M_{a}, M_{b})$. 
\end{proposition}

\subsection{Embeddings} \label{eq: P-embeddings}
We will need to consider certain spaces of embeddings of $\Sigma^{\ell}_{k}$-manifolds. 
For each $i \in \N$, we fix once
and for all integers $m_{i}$ with $m_{i} > p_{i} = \dim(P_{i})$, and smooth
embeddings
\begin{equation} \label{eq: P-embeddings 1} 
\phi_{i}: P_{i} \longrightarrow \R^{p_{i} + m_{i}},
\end{equation}
where recall, $\Sigma = (P_{1}, \dots, P_{i}, \dots)$ is the list of closed manifolds that we fixed in the beginning of Section \ref{section: The Category 1}. 
All of our constructions henceforth will depend on this choice of embeddings. 
\begin{Notation} \label{notation: embeddings} The following are some notational conventions that we will use throughout the rest of the paper when dealing with embedded $\Sigma^{\ell}_{k}$ manifolds. 
\begin{itemize}
\itemsep.3cm
\item
With $p_{i} = \dim(P_{i})$ and $m_{i}$ the integers chosen from (\ref{eq: P-embeddings 1}), we set
$$\bar{p} := \sum_{i\in \langle k \rangle}p_{i} \quad \text{and} \quad \bar{m} := \sum_{i\in \langle k \rangle}m_{i}.$$
\item 
For each subset $I \subseteq \langle k \rangle$, we denote 
$$p_{I} := \sum_{i\in I}p_{i} \quad \text{and} \quad m_{I} := \sum_{i\in I}m_{i}.$$
\item
If $n$ is any non-negative integer, we set
$\widehat{n} := n - \bar{p} - \bar{m}.$
\item
 For each subset $I \subseteq \langle k \rangle$, we set
 $$\R^{\bar{p}+\bar{m}}_{I} \; := \; \{(x_{1}, \dots, x_{k}) \in \R^{p_{1}+m_{1}}\times\cdots\times\R^{p_{k}+m_{k}} \; | \; \text{$x_{i} = 0$ if $i \notin I$}\}.$$ 
 \item
 Using the identification $\R^{\bar{p} + \bar{m}}_{I} = \prod_{i \in I}\R^{p_{i}+m_{i}}$,
 we let
$ 
\phi_{I}: P^{I} \longrightarrow \R^{\bar{p} + \bar{m}}_{I}
$
denote the embedding given by the product $\prod_{i\in I}\phi_{P_{i}}$. 
\item
For each $k \in \N$, we let $\R^{k}_{+}$ denote the product space $[0, \infty)^{k}$. 
For each subset $I \subseteq \langle k \rangle$, we let $\R^{k}_{+, I}$ denote the subspace of $\R^{k}_{+}$ given by 
$\{(x_{1}, \dots, x_{k}) \in \R^{k}_{+} \; | \; \text{$x_{i} = 0$ if $i \notin I$}\; \}.$
\end{itemize}
\end{Notation}
For what follows, let $X$ be a smooth manifold without boundary and let $M$ be a (possibly non-compact) $\Sigma^{\ell}_{k}$-manifold with $\partial_{0}M = \emptyset$. 
\begin{defn} \label{defn: i-P maps}
We define $\mathcal{E}_{\Sigma^{\ell}_{k}, n}(M, X)$ to be the space of smooth embeddings
$$f: M \longrightarrow X\times\R^{k}_{+}\times\R^{\widehat{n}}\times\R^{\bar{p}+\bar{m}}$$ 
that satisfy the following conditions:
\begin{enumerate} \itemsep.2cm
\item[i.]  For each subset $I \subseteq \langle k \rangle$, we have
$
f(\partial_{I}M) \subseteq \; X\times\R^{k}_{+, I^c }\times\R^{\widehat{n}}\times\R^{\bar{p}+\bar{m}}.  
$

\item[ii.] 
There exists a list of positive real numbers $\bar{\epsilon} = (\epsilon_{1}, \dots, \epsilon_{k})$ such that for every subset $I \subseteq \langle k \rangle$,
 the following diagram commutes
$$\xymatrix{
\partial_{I}M\times[0,\bar{\epsilon})^{k}_{I} \ar[d]^{\text{proj.}} \ar[rrrr]^{f\circ h_{I}} &&&& X\times\R^{k}_{+, I}\times(\R^{k}_{+, I^{c}}\times\R^{\widehat{n}}\times\R^{\bar{p}+\bar{m}}) \ar[d]^{\text{proj.}}\\
[0,\bar{\epsilon})^{k}_{I} \ar[rrrr]^{i_{I}} &&&& X\times\R^{k}_{+, I}}$$
where $i_{I}: [0, \bar{\epsilon})^{k}_{I} \hookrightarrow \R^{k}_{+, I}$ is the standard inclusion. 
In the upper-right corner we are using the identification,
$(\R^{k}_{+, I}\times\R^{k}_{+, I^{c}})\times\R^{\widehat{n}}\times\R^{\bar{p}+\bar{m}} \; = \; \R^{k}_{+}\times\R^{\widehat{n}}\times\R^{\bar{p}+\bar{m}}.$
 \item[iii.] For each subset $I \subseteq \langle \ell \rangle$, there is a factorization:
$f\mid_{\partial_{I}M} \; = \; f_{\beta_{I}M}\times \phi_{I}$
where the $\phi_{I}$ are the product embeddings specified in Notation Convention \ref{notation: embeddings} and 
$$f_{\beta_{I}M}: \beta_{I}M \longrightarrow  X\times\R^{k}_{+, I^{c}}\times\R^{\widehat{n}}\times\R^{\bar{p}+\bar{m}}_{I^{c}}$$
is an embedding which satisfies conditions i. and ii. given above.
\end{enumerate}
The space $\mathcal{E}_{\Sigma^{\ell}_{k}, n}(M, X)$ is topologized as a subspace of the space of smooth maps 
$$M \longrightarrow X\times\R^{k}_{+}\times\R^{\widehat{n}}\times\R^{\bar{p}+\bar{m}}$$  
in the $C^{\infty}$-topology.
We let $\mathcal{E}_{\Sigma^{\ell}_{k}, n}(M)$ denote the space $\mathcal{E}_{\Sigma^{\ell}_{k}, n}(M, \text{pt.})$, i.e.\ the space of embeddings $M \rightarrow \R^{k}_{+}\times\R^{\widehat{n}}\times\R^{\bar{p}+\bar{m}}$ that satisfy the conditions specified above. 
\end{defn} 

By the following proposition, we are justified in excluding the embeddings (\ref{eq: P-embeddings 1}) from the notation. 
This is proven in the same way as in \cite[Proposition 3.2]{P 13}.
\begin{proposition} \label{prop: independence of embedding}
Let $M$ be a $\Sigma^{\ell}_{k}$-manifold with $\partial_{0}M = \emptyset$, let $X$ be a smooth manifold, and let $n$ be positive integer. 
Then the homeomorphism type of the space $\mathcal{E}_{\Sigma^{\ell}_{k}, n}(M, X)$ does not depend on the embeddings $\phi_{i}: P_{i} \hookrightarrow \R^{p_{i}+m_{i}}$ used to define it. 
\end{proposition}

For each $n$ there is a natural embedding $\mathcal{E}_{\Sigma^{\ell}_{k}, n}(M, X) \hookrightarrow \mathcal{E}_{\Sigma^{\ell}_{k}, n+1}(M, X)$. 
We then define,
\begin{equation}
\xymatrix{
\mathcal{E}_{\Sigma^{\ell}_{k}}(M, X) := {\displaystyle \colim_{n\to\infty}}\;\mathcal{E}_{\Sigma^{\ell}_{k}, n}(M, X).
}
\end{equation}
Suppose that $M$ is a $\Sig{k}{\ell}$-manifold. 
Since $M$ is automatically a $\Sig{k}{\ell-1}$-manifold as well, the space 
$\mathcal{E}_{\Sig{k}{\ell-1}}(M)$
is defined.
We have a commutative diagram, 
\begin{equation} \label{eq: embeddings pull-back}
\xymatrix{
\mathcal{E}_{\Sig{k}{\ell}}(M) \ar[d]^{\beta_{\ell}} \ar@{^{(}->}[rr]  && \mathcal{E}_{\Sig{k}{\ell-1}}(M) \ar[d]^{\partial_{\ell}} \\
\mathcal{E}_{\Sig{k-1}{\ell-1}}(\beta_{\ell+1}M) \ar[rr]^{\times \phi_{\ell}} && \mathcal{E}_{\Sig{k-1}{\ell-1}}(\partial_{\ell+1}M)
}
\end{equation}
where the right vertical map sends an embedding
$g: M \rightarrow \R^{k}_{+}\times\R^{\infty}\times\R^{\bar{p}+\bar{m}}$
to its restriction $g\mid_{\partial_{\ell}M}$. 
The bottom horizontal map sends an embedding 
$$f: \beta_{\ell}M \longrightarrow \R^{k}_{+, \{\ell\}^{c}}\times\R^{\infty}\times\R^{\bar{p}+\bar{m}}_{\{\ell\}^{c}}$$
to the product embedding 
$$f\times \phi_{\ell}: (\beta_{\ell}M)\times P_{\ell}  \; \longrightarrow \; (\R^{k}_{+,\{\ell\}^{c}}\times\R^{\infty}\times\R^{\bar{p}+\bar{m}}_{\{\ell\}^{c}})\times\R^{\bar{p}+\bar{m}}_{\{\ell\}}.
$$ 
The top horizontal map is the inclusion. 
It follows easily that (\ref{eq: embeddings pull-back}) is a pull-back square.
The proof of the following lemma is given in Appendix \ref{embedding restriction}. 
\begin{lemma} \label{lemma: serre fibration embeddings 1} 
Let $M$ be a closed $\Sigma^{\ell}_{k}$-manifold. 
Then the restriction map 
 $$\partial_{\ell}: \mathcal{E}_{\Sig{k}{\ell-1}}(M) \; \longrightarrow \; \mathcal{E}_{\Sig{k-1}{\ell-1}}(\partial_{\ell}M)$$
 is a Serre-fibration.
 \end{lemma} 
The above lemma implies that the pull-back square (\ref{eq: embeddings pull-back}) is \textit{homotopy cartesian}. 
\begin{theorem} \label{thm: contractibility} 
Let $0 \leq \ell \leq k$ be integers and let $M$ be a closed $\Sigma^{\ell}_{k}$-manifold. 
Then the space $\mathcal{E}_{\Sig{k}{\ell}}(M)$ is weakly
contractible.
 \end{theorem} 
 \begin{proof}
 We prove this by induction on $k$. 
To prove the base case, let $k = 1$. 
It follows from \cite[Theorem 2.7]{G 08} that the spaces $\mathcal{E}_{\Sig{k}{0}}(M)$, $\mathcal{E}_{\Sig{k}{0}}(\partial_{1}M)$, and $\mathcal{E}_{\Sig{k}{0}}(\beta_{1}M)$ are all weakly contractible (these are spaces of \textit{neat embeddings of manifolds with corners}).
We then apply Lemma \ref{lemma: serre fibration embeddings 1} to see that the diagram 
$$\xymatrix{
\mathcal{E}_{\Sig{k}{1}}(M) \ar[d]^{\beta_{1}} \ar@{^{(}->}[rr]  && \mathcal{E}_{\Sig{k}{0}}(M) \ar[d]^{\partial_{1}} \\
\mathcal{E}_{\Sig{k-1}{0}}(\beta_{1}M) \ar[rr]^{\times \phi_{1}} && \mathcal{E}_{\Sig{k-1}{0}}(\partial_{1}M)
}$$
is a homotopy-cartesian square (this follows from the fact that the right-vertical map is a Serre-fibration and the fact that the square is already a pull-back). 
Since the bottom-left, bottom-right, and top-right spaces are all weakly contractible, it follows that the upper-left space is weakly contractible as well. 
This proves the base case of the induction. 
Now let $k \in \N$ be arbitrary, and suppose that $\mathcal{E}_{\Sig{k-1}{\ell}}(M)$ is weakly contractible for $\ell = 0, \dots, k$. 
By the Lemma \ref{lemma: serre fibration embeddings}, the diagram 
$$
\xymatrix{
\mathcal{E}_{\Sig{k}{\ell+1}}(M) \ar[d]^{\beta_{\ell+1}} \ar@{^{(}->}[rr]  && \mathcal{E}_{\Sig{k}{\ell}}(M) \ar[d]^{\partial_{\ell+1}} \\
\mathcal{E}_{\Sig{k-1}{\ell}}(\beta_{\ell+1}M) \ar[rr]^{\times \phi_{l+1}} && \mathcal{E}_{\Sig{k-1}{\ell}}(\partial_{\ell+1}M)
}
$$
is weakly homotopy-cartesian for all $\ell$. 
By the induction hypothesis, the bottom-left, bottom-right, and upper-right spaces are weakly contractible. 
It then follows that the upper-right space is weakly contractible as well. 
This completes the proof of the theorem.
\end{proof}
 We define similar embedding spaces for $\Sigma^{\ell}_{k}$-bordism triples. 
\begin{defn} \label{defn: space of embeddings morphisms} 
Let $(W, M_{a}, M_{b})$ be a $\Sigma^{\ell}_{k}$-bordism triple and let $n$ be a non-negative integer. 
For any smooth manifold without boundary $X$, we define $\mathcal{E}_{\Sigma^{\ell}_{k}, n}((W, M_{a}, M_{b}), X)$ to be the space of smooth embeddings 
   $$f: W \longrightarrow X\times[0, 1]\times \R^{k}_{+}\times\R^{\widehat{n}}\times\R^{\bar{p}+\bar{m}},$$
subject to the following conditions:
\begin{enumerate} \itemsep.2cm
\item[i.]   For each subset $I \subseteq \langle k \rangle$ we have, 
$$\xymatrix@-2pc@R-.25pc{
\; \; \; \; \; \;  f(\partial_{I}W) \subset \; X\times [0, 1]\times\R^{k}_{+, I^c}\times\R^{\widehat{n}}\times\R^{\bar{p}+\bar{m}},\\
f(M_{a}) \; \subset \; X\times\{0\}\times\R^{k}_{+}\times\R^{\widehat{n}}\times\R^{\bar{p}+\bar{m}}, \\ 
f(M_{b}) \; \subset \; X\times\{1\}\times\R^{k}_{+}\times\R^{\widehat{n}}\times\R^{\bar{p}+\bar{m}}.
}
$$
Furthermore, the restriction of $f$ to $M_{\nu}$, for $\nu = a, b$, is an element of the space 
$\mathcal{E}_{\Sigma^{\ell}_{k}, n}(M_{\nu}, X).$

\item[ii.] For each subset $I \subseteq \kset{k}$, the restriction map $f|_{\partial_{I}M}$ respects the collars in the same way as in Definition \ref{defn: i-P maps}. 
\item[iii.] For each subset $I \subseteq \langle \ell \rangle$, there is a
  factorization
$f\mid_{\partial_{I}M} \; = \; f_{\beta_{I}M}\times \phi_{I}$
just as in condition iii. of Definition \ref{defn: i-P maps}.
\end{enumerate}
\end{defn}
As before, we define
$$\mathcal{E}_{\Sigma^{\ell}_{k}}((W, M_{a}, M_{b}), X) := \colim_{n\to\infty}\mathcal{E}_{\Sigma^{\ell}_{k}, n}((W, M_{a}, M_{b}), X).$$
As before, we let $\mathcal{E}_{\Sigma^{\ell}_{k}}(W, M_{a}, M_{b})$ denote the space $\mathcal{E}_{\Sigma^{\ell}_{k}}((W, M_{a}, M_{b}), \text{pt.})$. 
The following theorem is proven in exactly the same way as Theorem \ref{thm: contractibility}. 
\begin{theorem} \label{theorem: weakly contractable, bordism space}
Let $(W, M_{a}, M_{b})$ be a $\Sigma^{\ell}_{k}$-cobordism triple. 
Then the space $\mathcal{E}_{\Sigma^{\ell}_{k}}(W, M_{a}, M_{b})$ is weakly contractable. 
\end{theorem}

The terminology given in the next definition will be useful later on when we define the cobordism category of $\Sigma^{\ell}_{k}$-manifolds and related constructions. 
 \begin{defn} \label{defn: Sigma-submanifold}
 Let $X$ be a smooth manifold without boundary. 
Let $M$ be a $\Sigma^{\ell}_{k}$-manifold with $\partial_{0}M = \emptyset$, that is embedded as a submanifold of $X\times\R^{k}_{+}\times\R^{\widehat{n}}\times\R^{\bar{p}+\bar{m}}$,
such that the inclusion map 
$$M \hookrightarrow X\times\R^{k}_{+}\times\R^{\widehat{n}}\times\R^{\bar{p}+\bar{m}}$$ 
is an element of the space $\mathcal{E}_{\Sigma^{\ell}_{k}, n}(M, X).$ 
In this case $M$ is called a \textit{$\Sigma^{\ell}_{k}$-submanifold over $X$}. 
In the case that $X$ is a point, we simply refer to $M$ as a \textit{$\Sigma^{\ell}_{k}$-submanifold}.
Similarly, let $(W, M_{a}, M_{b})$ be a $\Sigma^{\ell}_{k}$ bordism triple with $W$ embedded as a submanifold of $X\times[0,1]\times\R^{k}_{+}\times\R^{\widehat{n}}\times\R^{\bar{p}+\bar{m}}$ such that the inclusion
$W \hookrightarrow X\times[0,1]\times\R_{+}\times\R^{\widehat{n}}\times\R^{\bar{p}+\bar{m}}$
is an element of the space $\mathcal{E}_{\Sigma^{\ell}_{k}}((W, M_{a}, M_{b}), X)$. 
Then, $W$ is called a \textit{$\Sigma^{\ell}_{k}$-subcobordism over $X$}.
\end{defn}

\begin{remark}[Normal Bundles] \label{remark: normal bundles}
Let $M \subset X\times\R^{k}_{+}\times\R^{\widehat{n}}\times\R^{p+m}$ be a $\Sigma^{\ell}_{k}$-submanifold without boundary.  
Let $\pi: M \longrightarrow X$ denote the restriction to $M$ of the projection 
$$X\times\R^{k}_{+}\times\R^{\widehat{n}}\times\R^{\bar{p}+\bar{m}} \longrightarrow X.$$
It follows immediately from condition iii. of Definition \ref{defn: Sigma-submanifold} that $\pi$ is a smooth $\Sigma^{\ell}_{k}$-map (see Definition \ref{defn: sigma-k map}). 
Denote by $N \rightarrow M$ the normal bundle. 
The factorizations 
$$\partial_{I}M = \beta_{I}M\times \phi_{P_{I}}(P^{I})$$ 
with
$\beta_{I}M \subset  X\times\R^{k}_{+, I^{c}}\times\R^{\widehat{n}}$ and $\phi_{I}(P^{I}) \subset \R^{\bar{p}+\bar{m}}$, 
imply that the restriction of $N$ to $\partial_{I}M$ has the factorization 
\begin{equation}
N|_{\partial_{I}M} \; = \; (N_{\beta_{I}}\times N_{P^{I}})\oplus\epsilon^{|I|},
\end{equation}
where $N_{\beta_{I}} \rightarrow \beta_{I}M$ and $N_{P^{I}} \rightarrow \phi_{I}(P^{I})$ are the normal bundles for $\beta_{I}M$ and $\phi_{I}(P^{I})$ respectively, and $\epsilon^{|I|}$ is the trivial bundle with fibre-dimension $|I|$. 
\end{remark}

\subsection{Fibre bundles of $\Sigma^{\ell}_{k}$-manifolds}
Let $M$ be a closed $\Sigma^{\ell}_{k}$-manifold. 
Consider the space $\mathcal{E}_{\Sigma^{\ell}_{k}}(M)$. 
There is a free, continuous group action
$$\xymatrix{
\Diff_{\Sigma^{\ell}_{k}}(M)\times\mathcal{E}_{\Sigma^{\ell}_{k}}(M) \longrightarrow \mathcal{E}_{\Sigma^{\ell}_{k}}(M), \quad (g, \varphi) \mapsto \varphi\circ g.
}$$
We let $\mathcal{M}_{\Sigma^{\ell}_{k}}(M)$ denote the orbit space $\dfrac{\mathcal{E}_{\Sigma^{\ell}_{k}}(M)}{\Diff_{\Sigma^{\ell}_{k}}(M)}$. 
Similarly, if $(W, M_{a}, M_{b})$ is a $\Sigma^{\ell}_{k}$-manifold cobordism triple, there is a free, continuous group action 
$$\xymatrix{
\Diff_{\Sigma^{\ell}_{k}}(W, M_{a}, M_{b})\times\mathcal{E}_{\Sigma^{\ell}_{k}}(W, M_{a}, M_{b}) \longrightarrow \mathcal{E}_{\Sigma^{\ell}_{k}}(W, M_{a}, M_{b}), \quad (g, \varphi) \mapsto \varphi\circ g.
}$$
We let $\mathcal{M}_{\Sigma^{\ell}_{k}}(W, M_{a}, M_{b})$ denote the orbit space $\dfrac{\mathcal{E}_{\Sigma^{\ell}_{k}}(W, M_{a}, M_{b})}{\Diff_{\Sigma^{\ell}_{k}}(W, M_{a}, M_{b})}$. 
The following theorem is proven in exactly the same way as \cite[Lemma A.1]{P 13}. 
\begin{theorem} \label{theorem: local triviality}
The quotient maps
$$\xymatrix{
\mathcal{E}_{\Sigma^{\ell}_{k}}(M) \longrightarrow \mathcal{M}_{\Sigma^{\ell}_{k}}(M), \quad \text{and} \quad 
\mathcal{E}_{\Sigma^{\ell}_{k}}(W, M_{a}, M_{b}) \longrightarrow \mathcal{M}_{\Sigma^{\ell}_{k}}(W, M_{a}, M_{b})
}$$
are locally trivial fibre-bundles. 
\end{theorem}
\begin{remark}
By combining Theorem \ref{theorem: local triviality} with Lemma \ref{thm: contractibility} it follows that there are weak homotopy equivalences 
$$\xymatrix{
\mathcal{M}_{\Sigma^{\ell}_{k}}(M) \simeq B\Diff_{\Sigma^{\ell}_{k}}(M) \quad \text{and} \quad \mathcal{M}_{\Sigma^{\ell}_{k}}(W, M_{a}, M_{b}) \simeq B\Diff_{\Sigma^{\ell}_{k}}(W, M_{a}, M_{b})
}$$
where $B\Diff_{\Sigma^{\ell}_{k}}(M)$ and $B\Diff(W, M_{a}, M_{b})$ denote the \textit{classifying spaces} associated to the topological groups $\Diff_{\Sigma^{\ell}_{k}}(M)$ and $\Diff(W, M_{a}, M_{b})$. 
Furthermore, local triviality of the quotient maps from Theorem \ref{theorem: local triviality} implies that the spaces $\mathcal{M}_{\Sigma^{\ell}_{k}}(M)$ and $\mathcal{M}_{\Sigma^{\ell}_{k}}(W, M_{a}, M_{b})$ have the structure of \textit{infinite dimensional smooth manifolds} (see \cite{KM 97}). 
In this way it makes sense to speak of \textit{smooth maps} from a manifold into $\mathcal{M}_{\Sigma^{\ell}_{k}}(M)$ or $\mathcal{M}_{\Sigma^{\ell}_{k}}(W, M_{a}, M_{b})$.
\end{remark}

The following result is analogous to \cite[Lemma 3.5]{P 13} and is proven in the same way, using Theorem \ref{theorem: local triviality}.
\begin{lemma} \label{lemma: fibre-bundles}
Let $X$ be a smooth manifold without boundary and let $M$ be a closed $\Sigma^{\ell}_{k}$-manifold. 
There is a one-to-one correspondence between smooth maps, 
$$X \longrightarrow \mathcal{M}_{\Sigma^{\ell}_{k}}(M)$$ 
and closed $\Sigma^{\ell}_{k}$-submanifolds 
$$E \subset X\times\R^{k}_{+}\times\R^{\infty}\times\R^{\bar{p}+\bar{m}}$$
for some $n \in \N$, such that the projection $\pi: E \longrightarrow X$ is a smooth fibre-bundle bundle with fibre $M$ and structure group $\Diff_{\Sigma^{\ell}_{k}}(M)$. 

Similarly, if $(W, M_{a}, M_{b})$ is a $\Sigma^{\ell}_{k}$-cobordism triple, then there is a one-to-one correspondence between smooth maps, 
$$X \longrightarrow \mathcal{M}_{\Sigma^{\ell}_{k}}(W; M_{a}, M_{b})$$ 
and $\Sigma_{k}^{\ell}$-subcobordisms
$$E \subset X\times[0,1]\times\R^{k}_{+}\times\R^{\infty}\times\R^{\bar{p}+\bar{m}}$$
such that the projection $\pi: E \longrightarrow X$ is a smooth fibre-bundle bundle with fibre $W$ and structure group $\Diff_{\Sigma^{\ell}_{k}}(W, M_{a}, M_{b})$. 
\end{lemma}

\section{The cobordism category} \label{section: the cobordism category}
\label{subsection: The Cobordism Category}
We are now ready to define the category
$\mathbf{Cob}^{\Sig{k}{\ell}}_{d}$ that was roughly described in the introduction. 
An object of
$\mathbf{Cob}^{\Sig{k}{\ell}}_{d}$ is a pair $(M, a)$ with $a \in
\R$ and $M \subseteq
\R^{k}_{+}\times\R^{\infty}\times\R^{\bar{p}+\bar{m}}$ a
closed, $(d-1)$-dimensional $\Sig{k}{\ell}$-submanifold, as defined in Definition \ref{defn: Sigma-submanifold}.
A non-identity morphism from $(M_{a}, a)$ to $(M_{b}, b)$, is a triple $(W, a, b)$ with $a < b$, and where
$W \subset [a, b]\times\R^{k}_{+}\times\R^{\infty}\times\R^{\bar{p}+\bar{m}}$ is 
a $d$-dimensional $\Sig{k}{\ell}$ subcobordism such that 
$$W\cap(\{\nu\}\times\R^{k}_{+}\times\R^{\infty}\times\R^{\bar{p}+\bar{m}}) = M_{\nu} \quad \text{for $\nu = a, b$}.$$
Two morphisms $(W, a, b)$ and $(V, c, d)$ can be composed if $b = c$ and if
$$W\cap(\{b\}\times\R^{k}_{+}\times\R^{\infty}\times\R^{\bar{p}+\bar{m}}) \; = \; V\cap(\{c\}\times\R^{k}_{+}\times\R^{\infty}\times\R^{\bar{p}+\bar{m}}).$$
In this case, the composition of $(W, a, b,)$ of $(V, c, d)$ is given by $(W\cup V, a, d)$. 
Condition ii. of Definition \ref{defn: space of embeddings morphisms} (the condition requiring embeddings of $\Sigma^{\ell}_{k}$-manifolds to respect collars) ensures that the union $W\cup V$ is a smooth manifold with corners and so the composition is well defined. 
It is easy to check that this composition rule is associative. 

We want to make $\mathbf{Cob}^{\Sig{k}{\ell}}_{d}$ into a \textit{topological category}.  
Observe that as sets we have isomorphisms,
\begin{equation} \label{eq: topological structure}
\begin{aligned}
\Ob(\mathbf{Cob}^{\Sig{k}{\ell}}_{d}) \; &\cong \;   {\displaystyle \coprod_{M}}\ \ \ \  \mathcal{M}_{\Sigma^{\ell}_{k}}(M) \times \R, \\
\Mor(\mathbf{Cob}^{\Sig{k}{\ell}}_{d}) \; &\cong \;  \Ob(\mathbf{Cob}^{\Sig{k}{\ell}}_{d})\amalg {\displaystyle \coprod_{(W, M_{a}, M_{b})}}\mathcal{M}_{\Sig{k}{\ell}}(W, M_{a}, M_{b})\times\R^{2}_{<},
\end{aligned}
\end{equation}
where $M$ varies over diffeomorphism classes of closed
$(d-1)$-dimensional $\Sig{k}{\ell}$ manifolds, $(W, M_{a}, M_{b})$ varies over
diffeomorphism classes of $d$-dimensional $\Sig{k}{\ell}$-cobordism triples, and $\R^{2}_{<}$ denotes the space $\{(x, y) \in \R^{2} \; | \; x < y \}$. 
 We then topologize the category $\mathbf{Cob}^{\Sig{k}{\ell}}_{d}$ using the bijections of (\ref{eq: topological structure}). 
With $\mathbf{Cob}^{\Sig{k}{\ell}}_{d}$ topologized in this way, we see that composition and the target and source maps are all continuous.

Let $\ell \leq k$ be non-negative integers. 
Let $P$ be a closed manifold of dimension $p$ and let $\phi: P \longrightarrow \R^{p+m}$ be an embedding where $m > p$ is an integer. 
There is a continuous functor 
\begin{equation} \label{equation: times P map}
\tau_{P}: \mathbf{Cob}^{\Sigma^{\ell}_{k}}_{d} \longrightarrow \mathbf{Cob}^{\Sigma^{\ell}_{k}}_{d + p}
\end{equation}
defined by sending an object $M \subset \{a\}\times\R^{k}_{+}\times\R^{\infty}\times\R^{\bar{p}+\bar{m}}$ of $\mathbf{Cob}^{\Sigma^{\ell}_{k}}_{d}$ to the object of $\mathbf{Cob}^{\Sigma^{\ell}_{k}}_{d+p}$ given by the product $M\times\phi(P) \subset (\{a\}\times\R^{k}_{+}\times\R^{\infty}\times\R^{\bar{p}+\bar{m}})\times\R^{p+m}$. 
The definition of the $\tau_{P}$ on morphisms is similar.
An argument similar the proof of Proposition \ref{prop: independence of embedding} shows that since $m > p$, then the natural isomorphism class of the functor $\tau_{P}$ is independent of the chosen embedding $\phi$ (since $m > p$ implies that any two such embeddings are isotopic). 
We will latter need to use the following result. 
\begin{proposition} \label{proposition: homotopy invariance of times P}
Let $P$ and $P'$ be closed manifolds of dimension $p$ and let $\phi: P \rightarrow \R^{p+m}$ and $\phi': P' \rightarrow \R^{p +m}$ be embeddings with $m > p$. 
Suppose that $P$ and $P'$ are cobordant. 
Then the maps 
$
B\tau_{P}, \; B\tau_{P'}: B\mathbf{Cob}^{\Sigma^{\ell}_{k}}_{d} \longrightarrow B\mathbf{Cob}^{\Sigma^{\ell}_{k}}_{d+p},
$
induced by the functors $\tau_{P}$ and $\tau_{P'}$, are homotopic. 
\end{proposition}
\begin{proof}
Let $W$ be a cobordism from $P$ to $P'$, i.e.\ $W$ is a compact manifold with $\partial W = P\sqcup P'$. 
For each $a \in \R$, 
let $\Phi_{a}: W \longrightarrow [a, a+1]\times\R^{p+m}$ be a collared embedding such that the following diagram commutes,
$$\xymatrix{
P \ar[d]^{\phi} \ar@{^{(}->}[rr] && W \ar[d]^{\Phi} && P' \ar@{_{(}->}[ll] \ar[d]_{\phi'} \\
\{a\}\times\R^{p+m} \ar[rr]   && [a, a+1]\times\R^{p+m} && \{a+1\}\times\R^{p+m}. \ar[ll]
}$$
For each $t \in \R$, let 
$$T_{t}: \mathbf{Cob}^{\Sigma^{\ell}_{k}}_{d} \longrightarrow B\mathbf{Cob}^{\Sigma^{\ell}_{k}}_{d}$$
be the functor defined by $(M, a) \mapsto (M, t + a)$ and $(W, a, b) \mapsto (W, t + a, s + b)$. 
Clearly this transformation is an isomorphism with inverse given by the functor $T_{-t}$.
For any object $(M, a) \in \mathbf{Cob}^{\Sigma^{\ell}_{k}}_{d}$, the product 
\begin{equation} \label{equation: natural transformation}
M\times\Phi_{a}(W) \; \subset \; (\R^{k}_{+}\times\R^{\infty}\times\R^{\bar{p}+\bar{m}})\times([a, a+1]\times\R^{p+m})
\end{equation}
determines a morphism from $\tau_{P}(M, a)$ to $\tau_{P'}(M, a)$ in $\mathbf{Cob}^{\Sigma^{\ell}_{k}}_{d+p}$ (after identifying the ambient space $(\R^{k}_{+}\times\R^{\infty}\times\R^{\bar{p}+\bar{m}})\times[a, a+1]\times\R^{p+m}$ with $[a, a+1]\times\R^{k}_{+}\times(\R^{\infty}\times\R^{p+m})\times\R^{\bar{p}+\bar{m}}$.)
The correspondence 
$$(M, a) \mapsto \bigg(M\times\Phi_{a}(W): \tau_{P}(M, a) \rightarrow \tau_{P'}(M, a)\bigg)$$
defines a natural transformation 
$$\tau_{P}\longrightarrow T_{1}\circ\tau_{P'}$$
 and thus induces a homotopy between the maps $B\tau_{P}$ and $B\tau_{P'}$. 
 This concludes the proof of the proposition. 
\end{proof}

We will need to consider the functors
$$\xymatrix@C-.10pc@R-1.5pc{
\partial_{\ell+1}: \mathbf{Cob}^{\Sigma^{\ell}_{k}}_{d} \longrightarrow \mathbf{Cob}^{\Sigma^{\ell}_{k-1}}_{d-1}, &  (M, a) \mapsto (\partial_{\ell+1}M, a), \\
\beta_{\ell+1}: \mathbf{Cob}^{\Sigma^{\ell+1}_{k}}_{d} \longrightarrow \mathbf{Cob}^{\Sigma^{\ell}_{k-1}}_{d-1-p_{\ell+1}}, &  (M, a) \mapsto (\beta_{\ell+1}M, a),
}
$$
and the functor
$$\tau_{P_{\ell}}: \mathbf{Cob}^{\Sigma^{\ell}_{k}}_{d-p_{\ell+1} -1} \longrightarrow \mathbf{Cob}^{\Sigma^{\ell}_{k-1}}_{d-1}$$
from (\ref{equation: times P map}), defined using the embedding $\phi_{\ell}$ which was specified at the beginning of Section \ref{notation: embeddings}. 
It follows directly from the construction of the category $\mathbf{Cob}^{\Sigma^{\ell}_{k}}_{d}$, that the commutative square 
$$\xymatrix{
\mathbf{Cob}_{d}^{\Sigma^{\ell+1}_{k}} \ar[d]^{\beta_{\ell+1}} \ar[rrr] &&& \mathbf{Cob}_{d}^{\Sigma^{\ell}_{k}}
  \ar[d]^{\partial_{\ell + 1}}\\ 
  \mathbf{Cob}_{d-p_{\ell+1}-1}^{\Sigma^{\ell}_{k-1}}
  \ar[rrr]^{\tau_{P_{\ell+1}}} &&&
  \mathbf{Cob}_{d-1}^{\Sigma^{\ell}_{k-1}}
  }
$$ 
is a pull-back square (a pull-back in the category of topological categories), where the top-horizontal map is the natural inclusion. 
In Section \ref{section: A fibre sequence of classifying spaces}, we will prove that the induced commutative square obtained when passing to classifying spaces is homotopy-cartesian, which is the statement of Theorem \ref{thm: homotopy cartesian} from the introduction.

\section{A sheaf model for the cobordism category} \label{section: sheaf model}

\subsection{A recollection from \cite{GMTW 09} of sheaves} \label{section: A Recollection of Sheaves}
Let $\mathcal{X}$ denote the category with objects given by smooth manifolds without
boundary and with morphisms given by smooth maps. 
A sheaf
on $\mathcal{X}$ is defined to be a functor $\mathcal{F}: \mathcal{X}^{\op} \longrightarrow \textbf{Sets}$ which satisfies the following
condition. 
For any open covering $\{U_{i} \; | \; i \in \Lambda
\}$ of $X \in \Ob(\mathcal{X})$, and every collection $s_{i} \in
\mathcal{F}(U_{i})$ satisfying $s_{i}\mid_{U_{i}\cap U_{j}} =
s_{j}\mid_{U_{i}\cap U_{j}}$ for all $i, j \in \Lambda$, there is a
unique $s \in \mathcal{F}(X)$ such that $s\mid_{U_{i}} = s_{i}$ for
all $i \in \Lambda$. 

We are interested in the concordance theory of sheaves. 
Let $\mathcal{F}$ be a sheaf on $\mathcal{X}$. Two elements $s_{0}$
and $s_{1}$ of $\mathcal{F}(X)$ are said to be concordant if there
exists an element $s \in \mathcal{F}(X\times\R)$ that agrees with $\text{proj}_{X}^{*}(s_{0})$
in an open neighborhood of $X\times(-\infty, 0]$ and agrees with
$\text{proj}_{X}^{*}(s_{1})$ in an open neighborhood of $X\times[1,\infty)$.
We denote the set of concordance
classes of $\mathcal{F}(X)$ by $\mathcal{F}[X]$.  The correspondence
$X \mapsto \mathcal{F}[X]$ is clearly functorial in $X$. 

\begin{defn}For a sheaf $\mathcal{F}$ we define  the \textit{representing space}, denoted by $|\mathcal{F}|$, to be the geometric realization of the simplicial set given by the formula
$k \; \mapsto \; \mathcal{F}(\triangle^{k}_{e}),$
where 
$$\triangle^{k}_{e} := \{(x_{0}, x_{1}, \dots, x_{n}) \in \R^{n+1} \; | \; \sum x_{i} = 1\}$$ 
is the standard extended $k$-simplex. 
\end{defn}
From this definition it follows immediately that any natural transformation of sheaves
$\mathcal{F} \rightarrow \mathcal{G}$ induces a map between the
representing spaces $|\mathcal{F}| \rightarrow |\mathcal{G}|$.

In addition to set-valued sheafs, we will also have to consider
sheaves on $\mathcal{X}$ which take values in the category of small categories, which we denote by $\mathbf{Cat}$. 
A $\mathbf{Cat}$-valued
sheaf on $\mathcal{X}$ is a contravariant functor from $\mathcal{X}$
to $\mathbf{Cat}$ satisfying the same \textit{sheaf condition} given above. 
For a $\mathbf{Cat}$-valued sheaf $\mathcal{F}$ and for each positive integer $k$, one has 
the set valued sheaf
$\mathcal{N}_{k}\mathcal{F}$ defined by sending $X \in
\Ob(\mathcal{X})$ to the $k$-th nerve set of the category
$\mathcal{F}(X)$.
The correspondence,
$\ell \; \mapsto \; \mathcal{F}(\triangle^{\ell}_{e})$ defines a
simplicial category.  
Furthermore, the representing space 
$|\mathcal{F}|$ obtains the structure of a topological category with
$$\xymatrix{
\Ob(|\mathcal{F}|) \; = \; |\mathcal{N}_{0}\mathcal{F}|, & \Mor(|\mathcal{F}|) \; = \; |\mathcal{N}_{1}\mathcal{F}|.
}$$
One then has
\begin{equation} B|\mathcal{F}| \; \cong \; |k \mapsto \mathcal{N}_{k}\mathcal{F}(\triangle_{e}^{k})|, \end{equation}
where the space on the left is the classifying space of the topological category $|\mathcal{F}|$ and the right-hand side is the geometric realization of the simplicial space defined by the correspondence $k \mapsto \mathcal{N}_{k}\mathcal{F}(\triangle_{e}^{k})$. 

\subsection{The cobordism category sheaf} \label{subsection: the cobordism category sheaf}
We now define a $\mathbf{Cat}$-valued sheaf whose representing space is the topological category $\mathbf{Cob}^{\Sigma^{\ell}_{k}}_{d}$. 
The definitions of this section are analogous (and actually generalize) those in \cite[Section 2.3]{GMTW 09} and the results stated are analogous to those in \cite[Section 4]{GMTW 09} and are proven in exactly the same way. 
We omit the proofs of most of the propositions but we define everything explicitly.

Let $X \in \Ob(\mathcal{X})$. 
We will need to consider $\Sigma^{\ell}_{k}$-submanifolds (see Definition \ref{defn: Sigma-submanifold})
$$W \subset X\times\R\times(\R^{k}_{+}\times\R^{\infty}\times\R^{\bar{p}+\bar{m}})$$
over $X\times\R$, with $\partial_{0}W = \emptyset$. 
For such $\Sigma^{\ell}_{k}$-submanifolds, we will need to fix some notation.
\begin{Notation} The following notational conventions will be useful when dealing with sheaves on $\mathcal{X}$. 
\begin{itemize}
\itemsep.2cm
\item
Let $X \in \Ob(\mathcal{X})$ and let 
$W \subset X\times\R\times\R^{k}_{+}\times\R^{\infty}\times\R^{\bar{p}+\bar{m}}$
be a $\Sigma^{\ell}_{k}$-submanifold over $X\times \R$. 
We will denote by
\begin{equation} \label{equation: projection notation}
(\pi, f): W \longrightarrow X\times\R
\end{equation}
the restriction to $W$ of the projection, $\text{proj}_{X\times\R}: X\times\R\times\R^{k}_{+}\times\R^{\infty}\times\R^{\bar{p}+\bar{m}} \longrightarrow X\times\R$. 
\item
Let $a, b: X \longrightarrow \R$ be smooth functions with $a(x) < b(x)$ for all $x \in X$. 
We denote
$$\begin{aligned}
X\times (a, b) \;  &= \; \{(x, t) \in X\times\R \; | \; a(x) < t < b(x) \; \text{for all $x \in X$} \; \}, \\
X\times [a, b] \;  &= \; \{(x, t) \in X\times\R \; | \; a(x) \leq t \leq b(x) \; \text{for all $x \in X$}\; \}.
\end{aligned}
$$
\end{itemize}
\end{Notation}
The next definition is analogous to and generalizes \cite[Definition 2.6]{GMTW 09}.
\begin{defn} \label{defn: C transverse}
Let $X \in \Ob(\mathcal{X})$ and let $a, b: X \longrightarrow \R$ be smooth functions with $a(x) \leq b(x)$ for all $x \in X$.
For a positive real number $\varepsilon$, 
we define $\mathbf{C}^{\Sig{k}{\ell}, \pitchfork}_{d}(X; a, b, \varepsilon)$ to be the set of $(d + \dim(X))$-dimensional $\Sig{k}{l}$-submanifolds
$$W \; \subset \; X\times(a - \varepsilon, b + \varepsilon)\times (\R^{k}_{+}\times\R^{\infty}\times\R^{\bar{p}+\bar{m}})$$
over $X\times(a - \varepsilon, b + \varepsilon)$, subject to  the following conditions:
\begin{enumerate} \itemsep.2cm
\item[i.] The projection $\pi: W \longrightarrow X$ is a $\Sigma^{\ell}_{k}$-submersion. 
\item[ii.] The projection $(\pi, f): W \longrightarrow X\times(a - \varepsilon, b + \varepsilon)$ is a proper $\Sigma^{\ell}_{k}$-map.
\item[iii.] For $\nu = a, b$, the restriction of $\pi$ to the pre-image $(\pi, f)^{-1}(X\times(\nu - \varepsilon, \nu + \varepsilon))$ is a $\Sigma^{\ell}_{k}$-submersion (see Definition \ref{defn: sigma-k submersion}). 
\end{enumerate}
\end{defn}
We then eliminate dependence on $\varepsilon$ by setting
$\mathbf{C}^{\Sig{k}{\ell}, \pitchfork}_{d}(X; a, b) \; := \; {\displaystyle \colim_{\varepsilon \to 0}}\mathbf{C}^{\Sig{k}{\ell}, \pitchfork}_{d}(X; a, b, \varepsilon).$

\begin{defn} For $X \in \Ob(\mathcal{X})$, we set
$$\mathbf{C}^{\Sig{k}{\ell}, \pitchfork}_{d}(X) \; := \; {\displaystyle \coprod_{a \leq b}} \; \mathbf{C}^{\Sig{k}{\ell}, \pitchfork}_{d}(X; a, b)$$
with union ranging over all pairs of smooth functions $a, b: X \longrightarrow \R$ with $a \leq b$, such that the set
$\{x \in X \; | \; a(x) = b(x) \}$
is an open subset of $X$.
\end{defn}

The assignment $X \mapsto \mathbf{C}^{\Sig{k}{\ell}, \pitchfork}_{d}(X)$ for $X \in \mathcal{X}$ is a contravriant functor. Indeed, if $g: X \rightarrow Y$ is a smooth map, then for any $\Sig{k}{l}$-submanifold
 $$ W \subset  Y\times(a - \epsilon_{0}, b + \epsilon_{0})\times (\R^{k}_{+}\times\R^{\infty}\times\R^{\bar{p}+\bar{m}})$$
 representing an element of $\mathbf{C}^{\Sig{k}{\ell}, \pitchfork}_{d}(Y)$, it follows from condition iii. of Definition \ref{defn: C transverse} that the space $g^{*}(W)$ defined by forming the \textit{pull-back}, 
 $$\xymatrix{
 g^{*}(W) \ar[d]^{\hat{\pi}} \ar[rr]^{g^{*}} && W \ar[d]^{\pi}\\
 X \ar[rr]^{g} && Y}$$
 is a $\Sig{k}{\ell}$- submanifold over $X$. 
One can check also that $\mathbf{C}^{\Sig{k}{\ell}, \pitchfork}_{d}$ satisfies the sheaf condition on $\mathcal{X}$. 
Furthermore for each $X \in \Ob(\mathcal{X})$, the set $\mathbf{C}^{\Sig{k}{\ell}, \pitchfork}_{d}(X)$ is endowed with the structure of a category in the same way as described in \cite[Page 10]{GMTW 09}. Thus, $\mathbf{C}^{\Sig{k}{\ell}, \pitchfork}_{d}$ defines a $\mathbf{Cat}$-valued sheaf on $\mathcal{X}$. 
 \begin{defn} \label{defn: C sheaf} Let 
$\mathbf{C}^{\Sig{k}{\ell}}_{d}(X; a, b) \subseteq \mathbf{C}^{\Sig{k}{\ell}, \pitchfork}_{d}(X; a, b)$
be the subset satisfying the further condition:
\begin{enumerate}
\item[vi.] For $\nu = a, b$ and $x \in X$, let $J_{\nu}(x)$ be the interval $((\nu - \epsilon_{0})(x), (\nu + \epsilon_{0})(x)) \subseteq \R$ and let
$$V_{\nu} \; = \; (\pi, f)^{-1}(\{x\}\times J_{\nu}(x)) \subseteq \{x\}\times J_{\nu}(x)\times\R^{k}_{+}\times\R^{\infty}\times\R^{\bar{p}+\bar{m}}.$$
Then, there exists some $(d-1)$-dimensional $\Sigma_{k}^{\ell}$-submanifold,
$$M \subset \R^{k}_{+}\times\R^{\infty}\times\R^{\bar{p}+\bar{m}}$$
such that
$V_{\nu} \; = \; \{x\}\times J_{\nu} \times M \; \subset \; \{x\}\times J_{\nu}(x)\times\R^{k}_{+}\times\R^{\infty}\times\R^{\bar{p}+\bar{m}}.$
\end{enumerate}
\end{defn} 

We then proceed to define 
$$\mathbf{C}^{\Sig{k}{l}}_{d}(X) \; := \; \coprod_{a \leq b}\mathbf{C}^{\Sig{k}{l}}_{d}(X; a, b).$$
As with $\mathbf{C}^{\Sig{k}{l}, \pitchfork}_{d}$, the contravariant functor $X \mapsto \mathbf{C}^{\Sig{k}{l}}_{d}(X)$ defines a $\textbf{Cat}$-valued sheaf on $\mathcal{X}$. 
This added condition from Definition \ref{defn: C sheaf} implies that given 
$W \in  \mathbf{C}^{\Sig{k}{l}}_{d}(X; a, b)(X),$
for all $x \in X$ the inclusion map of the  
pre-mage 
$$(\pi, f)^{-1}(\{x\}\times[a(x), b(x)]) \; \subset \; [a(x), b(x)]\times\R^{k}_{+}\times\R^{\infty}\times\R^{\bar{p}+\bar{m}}$$
is $\Sigma_{k}^{\ell}$-subcobordism and is thus a an element of the morphism space $\Mor(\Cob_{d}^{\Sigma^{\ell}_{k}})$. 
Using the topological structure on $\mathbf{Cob}^{\Sigma^{\ell}_{k}}_{d}$ given in (\ref{eq: topological structure}), together with Lemma \ref{lemma: fibre-bundles} regarding fibre-bundles with $\Sigma^{\ell}_{k}$-manifold fibres, 
there is a natural isomorphism 
\begin{equation}
 C^{\infty}(\; \underline{\hspace{.3cm}} \;, \mathbf{Cob}^{\Sigma^{\ell}_{k}}_{d}) \stackrel{\cong} \longrightarrow \mathbf{C}^{\Sigma^{\ell}_{k}}_{d},
\end{equation}
 given by sending a smooth map $f: X \longrightarrow \mathbf{Cob}^{\Sigma^{\ell}_{k}}_{d}$ to the fibre-bundle of $\Sigma^{\ell}_{k}$-manifold cobordisms over $X$ that $f$ classifies. 
The topological structure on $\mathbf{Cob}^{\Sigma^{\ell}_{k}}_{d}$ defined in (\ref{eq: topological structure}) implies the following result (see also \cite[Proposition 2.9]{GMTW 09}).
\begin{proposition} \label{eq: singular set eq} 
There is a weak homotopy equivalence
$B|\mathbf{C}^{\Sig{k}{\ell}}_{d}| \; \simeq \; B\mathbf{Cob}^{\Sig{k}{\ell}}_{d}.$
\end{proposition} 
The next proposition is then proven in the same way as \cite[Proposition 4.4]{GMTW 09}.
\begin{proposition} \label{prop: pitchfork to non-pitchfork}
The natural transformation $\mathbf{C}^{\Sig{k}{\ell}}_{d} \longrightarrow \mathbf{C}^{\Sig{k}{\ell}, \pitchfork}_{d}$ induced by inclusion induces a weak homotopy equivalence $B|\mathbf{C}^{\Sig{k}{\ell}}_{d}| \stackrel{\simeq} \longrightarrow B|\mathbf{C}^{\Sig{k}{\ell}, \pitchfork}_{d}|$. 
\end{proposition}

We now define a $\mathbf{Set}$-valued sheaf whose representing space has the weak homotopy type of the classifying space $B\mathbf{Cob}^{\Sigma^{\ell}_{k}}_{d}$. 
  \begin{defn} \label{defn: D-sheaf} Let $X \in \Ob(\mathcal{X})$. 
We define $\mathbf{D}^{\Sig{k}{\ell}}_{d}(X)$ to be the set of $(\dim(X) + d)$-dimensional $\Sigma^{\ell}_{k}$-submanifolds  
$W \; \subset \; X\times \R\times (\R^{k}_{+}\times \R^{\infty}\times \R^{\bar{p}+\bar{m}})$
subject to the following conditions:
\begin{enumerate} \itemsep.2cm
\item[i.] The projection map $\pi: W \longrightarrow X$ is a $\Sigma^{\ell}_{k}$-submersion. 
\item[ii.] The projection $(\pi, f): W \longrightarrow X\times\R$ is a proper $\Sigma_{k}^{\ell}$-map.
\item[iii.] For any compact submanifold $K \subset X$, there exists $n \in \N$ such that 
$$\pi^{-1}(K) \subset K\times\R\times(\R^{k}_{+}\times \R^{n}\times \R^{\bar{p}+\bar{m}}).$$
\end{enumerate}
\end{defn}
It can be verified that $\mathbf{D}^{\Sig{k}{\ell}}_{d}$ satisfies the sheaf condition. 
We define a $\mathbf{Cat}$-valued sheaf $\mathbf{D}^{\Sig{k}{\ell}, \pitchfork}_{d}$ which can be compared directly to both $\mathbf{C}^{\Sig{k}{\ell}, \pitchfork}_{d}$ and $\mathbf{D}^{\Sig{k}{\ell}}_{d}$.

\begin{defn} For $X \in \Ob(\mathcal{X})$, we define $\Ob(\mathbf{D}^{\Sig{k}{\ell}, \pitchfork}_{d}(X))$ to be the set of pairs $(W, a)$ subject to the following conditions:
\begin{enumerate} \itemsep.2cm
\item[i.] $W \in  \mathbf{D}^{\Sig{k}{\ell}}_{d}(X)$, 
\item[ii.] $a: X \longrightarrow \R$ is a smooth function, 
\item[iii.] for each $x \in X$, the restriction map $f|_{\pi^{-1}(x)}: \pi^{-1}(x) \longrightarrow \R$ is transverse to $a(x) \in \R$. 
\end{enumerate}
The morphism set $\Mor(\mathbf{D}^{\Sig{k}{\ell}, \pitchfork}_{d}(X))$ is defined to be the set of triples $(W, a, b)$ such that 
\begin{enumerate} \itemsep.2cm
\item[(a)] $(W, a), (W, b) \in Ob(\mathbf{D}^{\Sig{k}{\ell}, \pitchfork}_{d}(X))$, 
\item[(b)] $a(x) \leq b(x)$ for all $x \in X$. 
\end{enumerate}
Two morphisms $(W, a, b), (V, c, d) \in \mathbf{D}^{\Sig{k}{\ell}, \pitchfork}_{d}(X)$ can be composed if and only if $V = W$ and $b(x) = c(x)$ for all $x \in X$. 
The composition of $(W, a, b)$ and $(V, c, d)$ is then given by $(W, a, d)$.  
\end{defn}
There is a natural transformation 
 $\gamma: \mathbf{D}^{\Sig{k}{\ell}, \pitchfork}_{d} \longrightarrow \mathbf{C}^{\Sig{k}{\ell}, \pitchfork}_{d}$
which is defined in exactly the same way as in \cite[Page 17]{GMTW 09}.
The following proposition is then proven in the same way as \cite[Proposition 4.3]{GMTW 09}.
\begin{proposition} \label{proposition: D' to C} The natural transformation $\gamma: \mathbf{D}^{\Sig{k}{\ell}, \pitchfork}_{d} \longrightarrow \mathbf{C}^{\Sig{k}{\ell}, \pitchfork}_{d}$ induces a weak homotopy equivalence 
$B|\gamma|: B|\mathbf{D}^{\Sig{k}{\ell}, \pitchfork}_{d}| \stackrel{\simeq} \longrightarrow B|\mathbf{C}^{\Sig{k}{\ell}, \pitchfork}_{d}|.$
\end{proposition}
Now, $\mathbf{D}^{\Sig{k}{\ell}}_{d}$ is a $\mathbf{Set}$-valued sheaf, however we may consider it to be a $\mathbf{Cat}$-valued sheaf by considering $\mathbf{D}^{\Sig{k}{\ell}}_{d}(X)$ to be the category with $\Ob(\mathbf{D}^{\Sig{k}{\ell}}_{d}(X)) = \mathbf{D}^{\Sig{k}{\ell}}_{d}(X)$ and where the only morphisms are the identity morphisms.  
Defined in this way, the representing space $|\mathbf{D}^{\Sig{k}{\ell}}_{d}|$ is a topological category with only identity morphisms and thus, there is a weak homotopy equivalence
$B|\mathbf{D}^{\Sig{k}{\ell}}_{d}| \simeq |\mathbf{D}^{\Sig{k}{\ell}}_{d}|.$
We have a natural transformation of $\mathbf{Cat}$-valued sheaves,
$$F: \mathbf{D}^{\Sig{k}{\ell}, \pitchfork}_{d} \longrightarrow \mathbf{D}^{\Sig{k}{\ell}}_{d}$$
determined by the rule 
$$(W, a) \mapsto W \quad \text{for $(W, a) \in \mathbf{D}^{\Sig{k}{\ell}}_{d}(X)$, \; $X \in \Ob(\mathcal{X})$.}$$
The following proposition is proven in the same way as \cite[Proposition 4.2]{GMTW 09}
\begin{proposition} \label{proposition: D' to D}
The natural transformation $F: \mathbf{D}^{\Sig{k}{\ell}, \pitchfork}_{d} \longrightarrow \mathbf{D}^{\Sig{k}{\ell}}_{d}$ induces a weak homotopy equivalence
$
B|F|: B|\mathbf{D}^{\Sig{k}{\ell}, \pitchfork}_{d}| \stackrel{\simeq} \longrightarrow  |\mathbf{D}^{\Sig{k}{\ell}}_{d}|.
$
\end{proposition} 

Combining Propositions \ref{eq: singular set eq}, \ref{prop: pitchfork to non-pitchfork}, \ref{proposition: D' to C} and \ref{proposition: D' to D}, we obtain a zig-zag of weak homotopy equivalences
\begin{equation} \label{zig zag 1}
\xymatrix{
B\mathbf{Cob}^{\Sigma^{\ell}_{k}}_{d} \ar[r]^{\simeq} & B|\mathbf{C}^{\Sig{k}{\ell}}_{d}| & B|\mathbf{C}^{\Sig{k}{\ell}, \pitchfork}_{d}| \ar[l]_{\simeq} &  B|\mathbf{D}^{\Sig{k}{\ell}, \pitchfork}_{d}| \ar[l]_{\simeq} \ar[r]^{\simeq} & |\mathbf{D}^{\Sig{k}{\ell}}_{d}|.
}
\end{equation}
\begin{corollary}
There is a weak homotopy equivalence $B\mathbf{Cob}^{\Sigma^{\ell}_{k}}_{d} \simeq |\mathbf{D}^{\Sig{k}{\ell}}_{d}|$. 
\end{corollary}

\section{A Fibre Sequence of Classifying Spaces} \label{section: A fibre sequence of classifying spaces}
\subsection{Fibre sequences}
We will need to use some further results from \cite[4.1.5]{MW 07} regarding the concordance theory of sheaves on $\mathcal{X}$. We recall now a certain property for sheaves, analogous to the \textit{covering homotopy property} for spaces. 
\begin{defn} 
A natural transformation of sheaves
$\alpha: \mathcal{F} \longrightarrow \mathcal{G}$
is said to have the \textit{concordance lifting property} if for any $X \in \Ob(\mathcal{X})$, the following condition holds: let 
$$s \in \mathcal{F}(X) \quad \text{and} \quad h \in \mathcal{G}(X\times \R)$$ 
be elements such that there exists $\epsilon > 0$ with
$$\text{proj}_{X}^{*}(\alpha(s))\mid_{X\times(-\infty, \epsilon)} \; = \; h \mid_{X\times(-\infty, \epsilon)}.$$
Then there exists $\widehat{h} \in \mathcal{F}(X\times \R)$ such that
$$\widehat{h}\mid_{X\times(-\infty, \widehat{\epsilon})} \; = \; \text{proj}_{X}^{*}(\alpha(s))\mid_{X\times(-\infty, \widehat{\epsilon})} \quad \text{and} \quad \alpha(\widehat{h}) \; = \; h,$$
where $\widehat{\epsilon}$ is some positive real number, possibly different than $\epsilon$.
\end{defn}
The main result that we will need to use is the following. 
\begin{proposition} \label{prop: concordance lifting property} 
Let $\mathcal{E}, \mathcal{F}$, and $\mathcal{G}$ be sheaves on $\mathcal{X}$ and let 
$u: \mathcal{E} \rightarrow \mathcal{G}$ and $v: \mathcal{F} \rightarrow \mathcal{G}$ be natural transformations. 
Let $\mathcal{E}\times_{\mathcal{G}}\mathcal{F}$ denote the sheaf defined by forming the fibred-product of $u$ and $v$. 
If $u$ has the concordance lifting property, then the projection $\mathcal{E}\times_{\mathcal{G}}\mathcal{F} \rightarrow \mathcal{F}$ has the concordance lifting property and the commutative square,
$$\xymatrix{
|\mathcal{E}\times_{\mathcal{G}}\mathcal{F}| \ar[rr] \ar[d] && |\mathcal{F}| \ar[d]^{|v|} \\
|\mathcal{E}| \ar[rr]^{|u|} && |\mathcal{G}|.
}$$
is homotopy cartesian. 
\end{proposition}
For $X \in \Ob(\mathcal{X})$, 
any element $z \in \mathcal{G}(\star)$ gives rise to an element (of the same name) $z \in \mathcal{G}(X)$, defined by pulling back $z \in \mathcal{G}(\star)$ over the constant map. 
\begin{defn} \label{defn: fibre sheaf} Let $\alpha: \mathcal{F} \longrightarrow \mathcal{G}$ and $z \in  \mathcal{G}(\star)$ be as above. 
The \textit{fibre} of the map $\alpha$ over $z$ is the sheaf $\Fib_{\alpha}^{z}$, defined by
$\Fib^{z}_{\alpha}(X) = \{ s \in \mathcal{F}(X) \; | \; \alpha(s) = z \}.$
\end{defn}
\begin{corollary} \label{concordance corollary} If $v: \mathcal{F} \longrightarrow \mathcal{G}$ has the concordance lifting property then for all $z \in  \mathcal{G}(\star)$, the sequence
$\xymatrix{
|\Fib^{z}_{\alpha}| \longrightarrow |\mathcal{F}| \stackrel{|v|} \longrightarrow |\mathcal{G}|
}$
is a homotopy fibre-sequence. 
\end{corollary}

\subsection{A fibre sequence of classifying spaces}

We now consider the natural transformation
\begin{equation}
\partial_{\ell}: \mathbf{D}^{\Sigma^{\ell-1}_{k}}_{d} \longrightarrow \mathbf{D}^{\Sigma^{\ell-1}_{k-1}}_{d-1},
\end{equation}
which is defined by sending an element $W \in \mathbf{D}^{\Sigma^{\ell-1}_{k}}_{d}(X)$ to the manifold 
$$\partial_{\ell}W \; \subset \; X\times\R\times\R^{k}_{+, \{\ell\}^{c}}\times\R^{\infty}\times\R^{\bar{p}+\bar{m}},$$
which determines an element in $\mathbf{D}^{\Sigma^{\ell-1}_{k-1}}_{d-1}(X)$.
We will need the following technical result. 

\begin{lemma} \label{lemma: bockstein k-l concordance} The natural transformation
$\partial_{1}: \mathbf{D}^{\Sig{k}{0}}_{d} \longrightarrow \mathbf{D}^{\Sig{k-1}{0}}_{d-1}$ 
has the concordance lifting
  property.
\end{lemma}
\begin{proof} 
Fix an element $X \in \Ob(\mathcal{X})$ and let $W \in
\mathbf{D}^{\Sig{k-1}{0}}_{d-1}(X\times\R)$ be a concordance. 
Let $V \in \mathbf{D}^{\Sig{k}{0}}_{d}(X)$ and $\varepsilon > 0$ be such that
\begin{equation} \label{eq: concrdance set-up} 
\text{proj}_{X}^{*}(\partial_{1}V)\mid_{X\times(-\infty, \varepsilon)} \; = \;
W\mid_{X\times(-\infty, \varepsilon)}.
\end{equation}
To prove the lemma we will construct a
concordance $\widetilde{V} \in
\mathbf{D}^{\Sig{k}{0}}_{d}(X\times\R)$ such that
$$ \partial_{1}\widetilde{V} \; = \; W \quad \text{and} \quad
\widetilde{V}\mid_{X\times(-\infty, \varepsilon)} \; = \;
\text{proj}_{X}^{*}(V)\mid_{X\times(-\infty, \varepsilon)}.$$ 

Since $V$ is a $\Sigma^{\ell}_{k}$-submanifold, if follows from condition ii. of Definition \ref{defn: i-P maps} that 
 $$\partial_{1}V \subset X\times\R^{k}_{+, \{1\}^{c}}\times\R^{\infty},$$
and that there exists a positive real number $\kappa$ such that 
$$
V \cap \bigg(X\times[0, \kappa)\times\R^{k}_{+, \{1\}^{c}}\times\R^{\infty}\bigg) \; = \; \partial_{1}V\times[0, \kappa).
$$
 Let $\varepsilon > 0$ be the real number from (\ref{eq: concrdance set-up}) and let
$\rho: \R\times [0,1) \longrightarrow \R\times [0, 1)$
be a smooth function that satisfies the following conditions: 
\begin{enumerate} \itemsep.3cm
\item[i.] The image of $\rho$ is contained in the subspace $\bigg(\R\times[0, \frac{2}{3})\bigg)\bigcup \bigg((-\infty, \frac{2\cdot\varepsilon}{3})\times [0,1)\bigg)$.
\item[ii.] $\rho$ is equal to the identity when restricted to the subspace  
$$\bigg(\R\times[0, \tfrac{1}{3})\bigg)\bigcup \bigg((-\infty, \tfrac{\varepsilon}{3})\times[0, 1)\bigg).$$

\item[iii.] Whenever $(t, s) \; \in \; [\frac{2\cdot\varepsilon}{3}, \infty)\times[\frac{1}{2}, 1)$, the equation $\rho(t, s) = (\frac{2\cdot\varepsilon}{3}, s)$ is satisfied.

\end{enumerate}

Let $\lambda: \R\times[0, 1) \longrightarrow \R$ be the projection $(t, s) \mapsto t$ where $(t, s) \in \R\times[0, 1)$. Using these two functions $\rho$ and $\lambda$, we define a new map
$$\widehat{\rho} : X\times\R\times[0, \kappa) \; \longrightarrow \; X\times\R$$
by the formula
\begin{equation} \label{eq: rho formula} (x, \; t, \; s) \; \mapsto \; (x, \; \; \lambda\circ \rho(t,  \; \tfrac{s}{\kappa}) ). \end{equation}

It follows directly from the definition of $\rho$ and $\lambda$ that for all $x \in X$, the following conditions are satisfied:
\begin{enumerate} \itemsep.3cm
\item[(a)] if $s \leq \frac{\kappa}{3}$ or $t \leq \frac{\varepsilon}{3}$, \; then \; $\widehat{\rho}(x, t, s) = (x, t)$, 
\item[(b)] if $(t, s) \in [\frac{2\cdot\varepsilon}{3}, \infty)\times[\tfrac{\kappa}{2}, \kappa)$, \;  then \; $\widehat{\rho}(x, t, s) = (x, \frac{2\cdot\varepsilon}{3})$.
\end{enumerate}
We now form the pull-back,
$$\xymatrix{
\widehat{\rho}^{*}(W) \ar[d]^{\widehat{\pi}} \ar[rr] && W \ar[d]^{\pi}\\
X\times\R\times[0, \epsilon_{1}) \ar[rr]^{\widehat{\rho}} && X\times\R}$$
where 
\begin{equation} \widehat{\rho}^{*}(W) = \label{eq: pull-back set} \{((x, t, s), w) \; \in \; X\times\R\times[0, \kappa)\times W\; \; | \; \; \widehat{\rho}(x, t, s) = \pi(w)\}. \end{equation} 
By definition of $\mathbf{D}^{\Sig{k-1}{0}}_{d-1}(X\times\R)$, $\pi$ is a $\Sigma^{0}_{k}$-submersion (see Definition \ref{defn: sigma-k submersion}).
It follows from the factorization in condition iii. of Definition \ref{defn: D-sheaf} that $\widehat{\rho}^{*}(W)$ is a $\Sig{k}{0}$-manifold.
The inclusion map 
$$W \hookrightarrow (X\times\R)\times\R\times\R^{k}_{+, \{1\}^{c}}\times\R^{\infty},$$
induces a natural embedding 
$$i: \widehat{\rho}^{*}(W) \; \hookrightarrow \; (X\times\R)\times\R\times[0,\kappa)\times\R^{k}_{+, \{1\}^{c}}\times\R^{\infty}.$$ 
We will denote by $W' \subset (X\times\R)\times\R\times[0,\kappa)\times\R^{k}_{+, \{1\}^{c}}\times\R^{\infty}$ the $\Sigma^{0}_{k}$-submanifold given by the image of this embedding. 
It can be verified directly using condition (b) and (\ref{eq: pull-back set}) that the intersections
$$W' \;  \cap \; \bigg[(X\times\R)\times\R\times[\tfrac{\kappa}{2},\kappa)\times\R^{k}_{+, \{1\}^{c}}\times\R^{\infty}\bigg],$$
$$\text{proj}_{X}^{*}V \; \cap \; \bigg[(X\times\R)\times\R\times[\tfrac{\kappa}{2},\kappa)\times\R^{k}_{+, \{1\}^{c}}\times\R^{\infty}\bigg],$$
are equal. 
Now, let $\widehat{V}$ denote the second of the above intersections. 
This implies that the subspace given by the union 
\begin{equation} \label{equation: concordance lift}
\widehat{V}\cup W' \subset (X\times\R)\times\R\times\R^{k}_{+}\times\R^{\infty}
\end{equation}
is a $\Sigma^{0}_{k}$-submanifold. 
Let 
$\widetilde{V}$ denote this $\Sigma^{0}_{k}$-submanifold given by the union in (\ref{equation: concordance lift}).
It follows that $\tilde{V} \in \mathbf{D}_{d}^{\Sigma^{0}_{k}}(X\times\R)$, that  $\partial_{1}\tilde{V} = W$, and that 
$$\tilde{V}|_{X\times(-\infty, \varepsilon)} \; = \; \text{proj}_{X}^{*}V|_{X\times(-\infty, \varepsilon)}.$$
This completes the proof of the lemma. 
\end{proof}
We now consider the natural transformation 
$\tau_{P_{\ell}}: \mathbf{D}^{\Sig{k-1}{\ell-1}}_{d-p_{\ell}-1} \longrightarrow \mathbf{D}^{\Sig{k-1}{\ell-1}}_{d}$
Defined by sending $W \in \mathbf{D}^{\Sig{k-1}{\ell-1}}_{d-p_{\ell}-1}(X)$ to the $d$-dimensional $\Sigma^{\ell}_{k}$-submanifold given by the product
$W\times\phi_{P_{\ell}}(P_{\ell}).$
Then consider the natural transformations 
$$\beta_{\ell}: \mathbf{D}^{\Sig{k}{\ell}}_{d} \longrightarrow \mathbf{D}^{\Sig{k-1}{\ell-1}}_{d-p_{\ell}-1}, \quad W \mapsto \beta_{\ell}W.$$ 
We then have:
\begin{corollary} \label{thm: homotopy cartesian sheaf}  For all $0 \leq \ell \leq k$, the commutative square
\begin{equation} \label{eq: cartesian rep space 2}
\xymatrix{
|\mathbf{D}^{\Sig{k}{\ell}}_{d}| \ar[rrr] \ar[d]^{|\beta_{\ell}|} &&& |\mathbf{D}^{\Sig{k}{\ell-1}}_{d}| \ar[d]^{|\partial_{\ell}|}\\
|\mathbf{D}^{\Sig{k-1}{\ell-1}}_{d-p_{\ell}-1}| \ar[rrr]^{|\tau_{P_{\ell}}|} &&& |\mathbf{D}^{\Sig{k-1}{\ell-1}}_{d-1}|
}
\end{equation}
is homotopy cartesian.
\end{corollary}
\begin{proof}
Consider the commutative diagram,
$$
\xymatrix{
\mathbf{D}^{\Sig{k}{1}}_{d} \ar[rrr] \ar[d]^{\beta_{1}} &&& \mathbf{D}^{\Sig{k}{0}}_{d} \ar[d]^{\partial_{1}}\\
\mathbf{D}^{\Sig{k-1}{0}}_{d-p_{\ell}-1} \ar[rrr]^{\tau_{P_{1}}} &&& \mathbf{D}^{\Sig{k-1}{0}}_{d-1}.}
$$
It follows by inspection that this is a pull-back square. 
By Lemma \ref{lemma: bockstein k-l concordance} it follows that the left-vertical map $\beta_{1}$ has the concordance lifting property.
It then follows from Proposition \ref{prop: concordance lifting property} that the lemma holds for $\ell = 1$. 
The general case of the lemma then follows by induction. 
\end{proof}

We obtain the following corollary, which is a restatement of Theorem \ref{thm: homotopy cartesian} from the introduction. 
\begin{corollary} The commutative square
$$
\xymatrix{
B\mathbf{Cob}_{d}^{\Sigma^{\ell+1}_{k}} \ar[rrr] \ar[d]^{B\beta_{\ell +1}} &&& B\mathbf{Cob}_{d}^{\Sigma^{\ell}_{k}} \ar[d]^{B\partial_{\ell+1}}\\
B\mathbf{Cob}_{d-p_{\ell+1}-1}^{\Sigma^{\ell}_{k-1}} \ar[rrr]^{B\tau_{P_{\ell+1}}} &&& B\mathbf{Cob}_{d-1}^{\Sigma^{\ell}_{k-1}}, 
}
$$
is homotopy-cartesian. 
\end{corollary}
\begin{proof}
The zig-zag of weak homotopy equivalences from (\ref{zig zag 1}) yields a commutative diagram,
 $$\xymatrix{ 
 & |\mathbf{D}^{\Sig{k-1}{\ell-1}}_{d}| \ar@{<~>}[dl]  \ar[dd] \ar[rr] && |\mathbf{D}^{\Sig{k}{\ell}}_{d}| \ar@{<~>}[dl]  \ar[dd]\\ 
 B\mathbf{Cob}_{d}^{\Sig{k-1}{\ell-1}} \ar[rr] \ar[dd] && B\mathbf{Cob}_{d}^{\Sig{k}{\ell}}  \ar[dd] &\\ 
 & |\mathbf{D}^{\Sig{k-1}{\ell-1}}_{d + p_{\ell}-1}|
  \ar@{<~>}[dl] \ar[rr] &&
  |\mathbf{D}^{\Sig{k-1}{\ell-1}}_{d-1}|
  \ar@{<~>}[dl]\\ B\mathbf{Cob}^{\Sig{k-1}{\ell-1}}_{d+p_{\ell}-1} \ar[rr] &&
 B\mathbf{Cob}^{\Sig{k-1}{\ell-1}}_{d-1}. &
  }$$ 
 By Corollary \ref{thm: homotopy cartesian sheaf}, the square in the back is homotopy cartesian. 
 Since all of the zig-zags are weak homotopy equivalences and the diagram commutes, it follows that the front square is homotopy-cartesian as well. 
 This completes the proof of the corollary.
\end{proof}

We now will identify the homotopy-fibres of the map $|\beta_{\ell}|$. 
The target space of the map $|\beta_{\ell}|: |\mathbf{D}^{\Sig{k}{\ell}}_{d}|  \longrightarrow |\mathbf{D}^{\Sig{k-1}{\ell-1}}_{d-p_{\ell}-1}|$ 
is not necessarily path-connected and thus it is not automatic that all fibres are homotopy equivalent. 
We will work to establish this fact.
The proof uses the following lemma.  
\begin{lemma}
The space $|\mathbf{D}^{\Sigma_{k}^{\ell}}_{d}|$ has the structure of a topological monoid that is associative up to homotopy. 
The induced monoid structure on $\pi_{0}(|\mathbf{D}^{\Sigma_{k}^{\ell}}_{d}|)$ is isomorphic to the bordism group $\Omega^{\Sigma^{k}_{\ell}}_{d-1}$. 
Furthermore, the map $|\beta_{\ell}|: |\mathbf{D}^{\Sig{k}{\ell}}_{d}| \longrightarrow |\mathbf{D}^{\Sig{k-1}{\ell-1}}_{d-p_{\ell}-1}|$ is a homomorphism with respect to the monoid structures.
\end{lemma}
\begin{proof}
The monoid structure on $|\mathbf{D}^{\Sigma^{\ell}_{k}}_{d+1}|$ is defined in exactly the same way as in \cite[Proposition 11.1]{P 13} (see also \cite[proof of Theorem 3.8]{MW 07}). 
Let
$\mathbf{D}^{\Sigma^{\ell}_{k}}_{d+1}\bar{\times}\mathbf{D}^{\Sigma^{\ell}_{k}}_{d+1}$
be the sheaf defined by letting
$$(\mathbf{D}^{\Sigma^{\ell}_{k}}_{d+1}\bar{\times}\mathbf{D}^{\Sigma^{\ell}_{k}}_{d+1})(X) \; \subset \; \mathbf{D}^{\Sigma^{\ell}_{k}}_{d+1}(X)\times\mathbf{D}^{\Sigma^{\ell}_{k}}_{d+1}(X)$$
be the subset which consists of all pairs $(W_{1}, W_{2})$ that are disjoint. 
There is a map,
\begin{equation} \label{eq: partial product} 
\mu:
(\mathbf{D}^{\Sigma^{\ell}_{k}}_{d+1}\bar{\times}\mathbf{D}^{\Sigma^{\ell}_{k}}_{d+1})(X)
\longrightarrow \mathbf{D}^{\Sigma^{\ell}_{k}}_{d+1}(X), \quad (W_{1}, W_{2}) \mapsto W_{1}\sqcup W_{2}.
\end{equation}
This map yields a partially defined product on $\mathbf{D}^{\Sigma^{\ell}_{k}}_{d+1}$ which is clearly associative, commutative, and the identity element is given by the empty set. 
Furthermore, it follows from a general position argument that the inclusion map 
$ \mathbf{D}^{\Sigma^{\ell}_{k}}_{d+1}\bar{\times}\mathbf{D}^{\Sigma^{\ell}_{k}}_{d+1} \longrightarrow \mathbf{D}^{\Sigma^{\ell}_{k}}_{d+1}\times\mathbf{D}^{\Sigma^{\ell}_{k}}_{d+1}$
is a weak equivalence of sheaves. 
From this weak equivalence, the representing space $|\mathbf{D}^{\Sigma^{\ell}_{k}}_{d+1}|$ inherits a monoid structure defined by
$$\xymatrix{
|\mathbf{D}^{\Sigma^{\ell}_{k}}_{d+1}|\times|\mathbf{D}^{\Sigma^{\ell}_{k}}_{d+1}| \ar[r]^{\simeq} & |\mathbf{D}^{\Sigma^{\ell}_{k}}_{d+1}\bar{\times}\mathbf{D}^{\Sigma^{\ell}_{k}}_{d+1}| \ar[rr]^{\mu} && |\mathbf{D}^{\Sigma^{\ell}_{k}}_{d+1}|,
}$$
which is associative and commutative up to homotopy and has strict identity. 
It follows that the natural transformation $\beta_{\ell}: \mathbf{D}^{\Sig{k}{\ell}}_{d} \longrightarrow \mathbf{D}^{\Sig{k-1}{\ell-1}}_{d-p_{\ell}-1}$ respects this product and thus the induced map $|\beta_{\ell}|: |\mathbf{D}^{\Sig{k}{\ell}}_{d}| \longrightarrow |\mathbf{D}^{\Sig{k-1}{\ell-1}}_{d-p_{\ell}-1}|$ is a homomorphism of homotopy monoids. 

We define a monoid homomorphism $\alpha: \pi_{0}(\mathbf{D}^{\Sigma^{\ell}_{k}}_{d}) \longrightarrow \Omega^{\Sigma^{k}_{\ell}}_{d-1}$ in a way similar to as in \cite[Lemma 3.2]{H 14}. 
For any $\Sigma^{\ell}_{k}$-submanifold $W \subset \R\times\R^{k}_{+}\times\R^{\infty}\times\R^{\bar{p}+\bar{m}}$, we may perturb $W$ through a small ambient isotopy (through $\Sigma^{k}_{\ell}$-submanifolds) to a new $\Sigma^{\ell}_{k}$ submanifold $W'$, such that the restriction map $f: W' \longrightarrow \R$ is $\Sigma^{\ell}_{k}$-transverse (see Definition \ref{defn: transversality}) to $0 \in \R$, where $f$ is the restriction to $W'$ of the projection 
$\R\times\R^{k}_{+}\times\R^{\infty}\times\R^{\bar{p}+\bar{m}} \rightarrow \R.$
We then define $\alpha(W)$ to be the $\Sigma^{\ell}_{k}$-bordism class in $\Omega^{\Sigma^{\ell}_{k}}_{d-1}.$ 
It is immediate that this map is a homomorphism. 
The inverse to $\alpha$ is defined by sending a bordism class $[M] \in \Omega^{\Sigma^{\ell}_{k}}_{d-1}$ to the concordance class represented by the element 
$\R\times \varphi(M) \; \subset \; \R\times(\R^{k}_{+}\times\R^{\infty}\times\R^{\bar{p}+\bar{m}}),$
where $\varphi: M \longrightarrow \R^{k}_{+}\times\R^{\infty}\times\R^{\bar{p}+\bar{m}}$ is some element of the embedding space $\mathcal{E}_{\Sigma^{\ell}_{k}}(M)$. 
It is an easy exercise to see that this yields a well-defined map (see \cite[Lemma 3.2]{H 14} for details). 
Since $\Omega^{\Sigma^{\ell}_{k}}_{d-1}$ is a group, it follows from that fact that $\alpha$ is a bijective homomrphism that $\pi_{0}(\mathbf{D}^{\Sigma^{\ell}_{k}}_{d})$ is a group as well. 
This completes the proof of the lemma. 
\end{proof}

The next corollary follows from the fact that the map $|\beta_{\ell}|: |\mathbf{D}^{\Sig{k}{\ell}}_{d}| \longrightarrow |\mathbf{D}^{\Sig{k-1}{\ell-1}}_{d-p_{\ell}-1}|$ is a homomorphism of \textit{group-like monoids}.
\begin{corollary} \label{corollary: equivalent fibres}
All non-empty fibres of the map $|\beta_{\ell}|: |\mathbf{D}^{\Sig{k}{\ell}}_{d}| \longrightarrow |\mathbf{D}^{\Sig{k-1}{\ell-1}}_{d-p_{\ell}-1}|$ are homotopy equivalent. 
\end{corollary}

The element $\emptyset \in \mathbf{D}^{\Sig{k-1}{\ell-1}}_{d-1}(\; \star \;)$, given by the empty set, determines an element
$\emptyset \in \mathbf{D}^{\Sig{k-1}{\ell-1}}_{d-1}(X)$ for all $X \in \Ob(\mathcal{X})$. 
Recall from Definition \ref{defn: fibre sheaf} the fibre-sheaf defined by,
$$\xymatrix{
\Fib_{\beta_{\ell}}^{\emptyset}(X) \; = \; \{ W \in \mathbf{D}^{\Sig{k}{\ell}}_{d}(X) \; | \; \beta_{\ell}(W) = \emptyset \; \},
}$$
for $X \in \Ob(\mathcal{X})$. 

\begin{proposition} The fibre sheaf $\Fib_{\beta_{\ell}}^{\emptyset}$ is isomorphic to $\mathbf{D}^{\Sig{k-1}{\ell-1}}_{d}$. \end{proposition}
\begin{proof}
For  $X \in \Ob(\mathcal{X})$, an element of $\Fib_{\beta_{\ell}}^{\emptyset}(X)$ is given by a $\Sig{k}{\ell}$-submanifold
$$W \; \subset \; X\times\R\times\R^{k}_{+}\times\R^{\infty}\times\R^{\bar{p}+\bar{m}}$$
satisfying all conditions of the definition of $\mathbf{D}^{\Sig{k}{\ell}}_{d}(X)$, with the added property that 
$$W \cap (X\times\R\times\R^{k}_{+, \{k\}^{c}}\times\R^{\infty}\times\R^{\bar{p}+\bar{m}}) \; = \; \partial_{\ell}W \; = \; \emptyset.$$
The proof of the proposition follows immediately from this observation and the definition of a $\Sigma^{\ell-1}_{k-1}$-manifold. 
\end{proof}

\section{A Cubical Diagram of Thom-Spectra} \label{section: k spaces}
In this section we construct the spectrum which appears in the statement of Theorem \ref{thm: weak homotopy equivalence}. 
We must first cover some preliminaries on cubical diagrams of spaces and spectra. 
\subsection{$k$-Cubic Spaces} \label{section: k- Spaces}
 As in previous sections, let $\langle k \rangle$ denote the set $\{1,
 \cdots, k\}$. 
 We denote by $2^{\langle k \rangle}$
 the category with objects given by the subsets of $\langle k \rangle$, and with morphisms given by the inclusion maps. 
  We call a functor from
 $(2^{\langle k \rangle})^{\text{op}}$ to $\mathbf{Top}$ (the category of
 topological spaces) a \textit{$k$-cubic space}. 
 Such functors will usually be denoted by 
 $$X_{\bullet}: (2^{\langle k \rangle})^{\text{op}} \longrightarrow \Top, \quad J \mapsto X_{J}.$$
 In order to define a $k$-cubic space, one needs to associate to each subset $J \subseteq \langle k \rangle$ a space $X_{J}$ and to pairs of subsets $I \subseteq J \subseteq \langle k \rangle$ maps 
$f_{J, I}: X_{J} \longrightarrow X_{I}$,
such that for any triple of subsets $K \subseteq I \subseteq J \subseteq \langle k \rangle$, the equation $f_{J, K} = f_{I, K}\circ f_{J, I}$ holds, and $f_{I, I} = Id_{X_{I}}$. 
We will sometimes refer to the spaces $X_{I}$ as \textit{vertices} and the maps $f_{J, I}: X_{J} \longrightarrow X_{I}$ as \textit{edges}.

 A \textit{morphism} of $k$-cubic spaces $F_{\bullet}: X_{\bullet} \longrightarrow Y_{\bullet}$ is defined to be a natural transformation of the functors $X_{\bullet}$ and $Y_{\bullet}$. 
 We denote the space of all such $k$-cubic space maps by 
$\Maps_{\kset{k}}(X_{\bullet}, \; Y_{\bullet}).$
This space is topologized naturally as a subset of the product
$\prod_{J\subseteq \kset{k}}\Maps(X_{J}, Y_{J})$,
where $\Maps(X_{J}, Y_{J})$ is the space of continuous maps from $X_{J}$ to $Y_{J}$, topologized in the \textit{compact-open} topology.

A $\Sigma_{k}$-manifold $W$ determines a $k$-cubic space by the correspondence $I \mapsto \partial_{I}W$.
Also notice that the spaces $\R^{k}_{+}$ and $\R^{\bar{p} + \bar{m}}$ from the previous sections determine $k$-cubic spaces via the correspondences,
$$\begin{aligned}
I \; \mapsto \; \R^{k}_{+, I} \; &= \; \{(t_{1}, \dots, t_{k}) \; | \; t_{i} = 0 \; \; \text{if $i \notin I$} \},\\
I \; \mapsto \; \R^{\bar{p} + \bar{m}}_{I} \; &= \; \{(x_{1}, \dots , x_{k}) \in \R^{p_{1} + m_{1}}\times\cdots \times\R^{p_{k}+m_{k}} \; | \; x_{i} = 0 \; \; \text{if $i \notin I$} \}.
\end{aligned}$$
In addition to $k$-cubic spaces we will also have to consider $k$-cubic spectra. 
A $k$-cubic spectrum is a functor 
$\mathsf{X}_{\bullet}: (2^{\kset{k}})^{\text{op}} \longrightarrow \Spec$
where $\Spec$ is the category of spectra. 
It is required that for each pair of subsets $I \subseteq J \subseteq \kset{k}$, the associated map
$\mathsf{X}_{J} \longrightarrow \mathsf{X}_{I}$
is a \textit{strict} map of spectra of degree $0$. 

Let $\mathsf{X}_{\bullet}$ be a $k$-cubic spectrum. 
Then for each integer $n$, there is a $k$-cubic space $(\mathsf{X}_{\bullet})_{n}$ defined by sending each $J \subseteq \kset{k}$ to the $n$th space of the spectrum $\mathsf{X}_{J}$. 
The operations of suspending $\mathsf{X}_{\bullet} \mapsto \Sigma\mathsf{X}_{\bullet}$, and de-suspending $\mathsf{X}_{\bullet} \mapsto \Sigma^{-1}\mathsf{X}_{\bullet}$ are defined in the obvious way.

\subsection{The total homotopy cofibre of a $k$-cubic space} \label{subsection: total cofibres}
We define an important type of homotopy colimit associated to a $k$-cubic space called the \textit{total homotopy cofibre}. 
We first introduce some new notation. 
For $j = 1, \dots, k$, consider the function
$$
\sigma_{j}: \langle k-1 \rangle \longrightarrow \langle k \rangle, \quad \quad
\sigma(i) = 
\begin{cases}
i &\quad \text{if $i < j$,}\\
i + 1 &\quad \text{if $i \geq j$.}
\end{cases}
$$
We then define functors 
\begin{equation}
\begin{aligned}
\partial_{j}: 2^{\kset{k-1}} \longrightarrow 2^{\kset{k}}, & \quad \{i_{1}, \dots, i_{n}\} \mapsto \{\sigma_{j}(i_{1}), \dots, \sigma_{j}(i_{n})\}\cup\{j\},\\
\bar{\partial}_{j}: 2^{\kset{k-1}} \longrightarrow 2^{\kset{k}}, & \quad \{i_{1}, \dots, i_{n}\} \mapsto \{\sigma_{j}(i_{1}).\dots, \sigma_{j}(i_{n})\}.
\end{aligned}
\end{equation}
For any $k$-cubic space $X_{\bullet}: (2^{\kset{k}})^{\text{op}} \; \longrightarrow \; \Top$, 
we define $X_{\bullet, \; j}$ and $X_{\bullet, \; \bar{j}}$ to be the $(k-1)$-cubic spaces obtained from $X_{\bullet}$ by precomposing $X_{\bullet}$ with the functors $\partial_{j}$ and $\bar{\partial}_{j}$ respectively. 
There is a natural map of $(k-1)$-cubic spaces
\begin{equation} \label{eq: k-1 space map} X_{\bullet, j} \; \longrightarrow X_{\bullet, \bar{j}}, \end{equation} 
induced by the maps $X_{I\cup\{j\}} \longrightarrow X_{I}$, for subsets $I \subseteq \langle k \rangle\setminus\{j\}$. 

\begin{defn} \label{defn: total homotopy cofibre}
 Let $X_{\bullet}$ be a $k$-cubic space. We define the $\textit{Total Homotopy Cofibre}$ of $X_{\bullet}$, which we denote by 
$\tCofibre_{\kset{k}}X_{\bullet},$
inductively on $k$ as follows. 
If $k = 1$, we define
$\tCofibre_{\kset{1}}X_{\bullet}$ to be the homotopy cofibre (or mapping-cone) of the natural map $X_{\{1\}} \longrightarrow X_{\emptyset}$. 
Assume now that the total homotopy cofibre is defined for all $(k-1)$-cubic spaces. 
The map from (\ref{eq: k-1 space map}) induces a map, 
$$\xymatrix{
\tCofibre_{\kset{k-1}} X_{\bullet, k} \longrightarrow \tCofibre_{\kset{k-1}} X_{\bullet, \bar{k}}.
}$$
Using this we define
\begin{equation} \label{eq: inductive t cofibre} 
\xymatrix{
\tCofibre X_{\bullet} \; := \; \Cofibre\bigg(\tCofibre_{\kset{k-1}} X_{\bullet, k} \longrightarrow \tCofibre_{\kset{k-1}} X_{\bullet, \bar{k}}\bigg). 
}
\end{equation}
\end{defn}
The total homotopy cofibre of a $\langle k \rangle$-cubic spectrum is defined similarly. 

\subsection{Loop Spaces} \label{subsection: Loop Spaces} In this section we describe a useful mapping space associated to a $\langle k \rangle$-cubic space. 
 \begin{defn} For $n \geq 0$, let $D^{n, \langle k \rangle}_{\bullet}$ denote the $k$-cubic space defined by setting
 $$D^{n, \langle k \rangle}_{J} \; := \; (\R^{k}_{+, J^{c}}\times\R^{n})^{c},$$
 where the super-script $c$ signifies the \textit{one-point compactification}. The maps $D^{n, \langle k \rangle}_{J} \rightarrow D^{n, \langle k \rangle}_{I}$ for pairs of subsets $I \subseteq J \subseteq \langle k\rangle$ are given by inclusion. \end{defn}
For a $k$-cubic space $X_{\bullet}$ we define, 
 \begin{equation}
 \xymatrix{
   \Omega^{n+k}_{\langle k \rangle}X_{\bullet} \; := \; \Maps_{\langle k \rangle}(D^{n, \langle k \rangle}_{\bullet}, \; X_{\bullet}).
   } \end{equation}
Let $\mathsf{X}_{\bullet}$ be a $k$-cubic
 spectrum. 
 For integers $m$ and $n$ there are maps
 $$\Omega^{n+k}_{\langle k \rangle}(\mathsf{X}_{\bullet})_{n+k+m} \;
 \longrightarrow \; \Omega^{n+k+1}_{\langle k
   \rangle}(\mathsf{X}_{\bullet})_{n+k+m+1}$$ induced by the structure maps of the spectrum. 
 Using these maps we define,
 \begin{equation} \Omega^{\infty - m}_{\langle k \rangle}\mathsf{X}_{\bullet} \; := \; \colim_{n\to\infty}\Omega^{n+k}_{\langle k \rangle}(\mathsf{X}_{\bullet})_{n+k+m}. \end{equation}
For each pair of integers $n$ and $m$, the natural maps
$$\xymatrix{
(\mathsf{X}_{J})_{n+k+m} \; \longrightarrow \; (\tCofibre_{\kset{k}}\mathsf{X}_{\bullet})_{n+k+m}}$$ 
induce a map
\begin{equation} \label{eq: limit natural map}
\xymatrix{ \Omega^{\infty-m}_{\langle k \rangle}\mathsf{X}_{\bullet} \; \longrightarrow \; \Omega^{\infty-m}\tCofibre_{\kset{k}}\mathsf{X}_{\bullet}}, 
\end{equation}
where the space on the right is the infinite loopspace associated to the spectrum $\Sigma^{-m}\tCofibre_{\kset{k}}\mathsf{X}_{\bullet}$. 
The following result is essentially the same as \cite[Proposition 3.15]{G 08} (see also \cite[Lemmas 3.1.2 and 3.1.4]{L 00}).
\begin{theorem} \label{thm: loopspace eq 1} 
For all $m$, the map from {\rm (\ref{eq: limit natural map})} yields a
homotopy equivalence,
 $$\xymatrix{
 \Omega^{\infty-m}_{\langle k \rangle}\mathsf{X}_{\bullet} \; \simeq \; \Omega^{\infty-m}\tCofibre_{\kset{k}}\mathsf{X}_{\bullet}.
 }$$
 \end{theorem}
 \begin{proof}[Proof sketch]
 This is proven by induction on $k$. 
Let $k = 0$. We have $D^{n, \langle 0 \rangle} \cong S^{n}$ for all $n \in \N$. 
It follows then from Definition \ref{defn: total homotopy cofibre} that if $\mathsf{X}_{\bullet}$ is a $0$-cubic spectrum (which is just a spectrum), there is a homeomorphism,
$\Omega^{\infty-m}_{\langle 0 \rangle}\mathsf{X}_{\bullet} \;  \cong \; \Omega^{\infty-m}\mathsf{X}_{\emptyset}.$
This proves the base case. 
Now suppose that the lemma holds for all $(k-1)$-cubic spectra. 
The result then follows from considering the map of fibre sequences
$$\xymatrix{
 \Omega^{\infty-m}_{\langle k-1 \rangle}\mathsf{X}_{\bullet, k} \ar[d] \ar[rr]^{\simeq\ \ \ \ \ \ \ \ } && \Omega^{\infty-m}\tCofibre_{\kset{k-1}}\mathsf{X}_{\bullet, k} \ar[d]\\
  \Omega^{\infty-m}_{\langle k-1 \rangle}\mathsf{X}_{\bullet, \bar{k}} \ar[d] \ar[rr]^{\simeq\ \ \ \ \ \ \ \ } && \Omega^{\infty-m}\tCofibre_{\kset{k-1}}\mathsf{X}_{\bullet, \bar{k}} \ar[d]\\
   \Omega^{\infty-m}_{\langle k \rangle}\mathsf{X}_{\bullet} \ar[rr] && \Omega^{\infty-m}\tCofibre_{\kset{k}}\mathsf{X}_{\bullet},\\
}$$
 where the top two horizontal maps are weak homotopy equivalences by the induction hypothesis. 
 \end{proof}

\subsection{A $k$-cubic Thom-spectrum} \label{subsection: k-cube of thom spectra}
We now construct the $k$-cubic spectrum 
$\MT_{\Sigma^{\ell}_{k}}(d)_{\bullet}$ from the statement of Theorem \ref{thm: weak homotopy equivalence} in the introduction.
This $k$-cubic spectrum will be defined in such a way so that for each subset $J \subseteq \langle k \rangle$, 
$$\MT_{\Sigma^{\ell}_{k}}(d)_{J} = \Sigma^{-|J|}\MT(d-p_{J}-|J|),$$
where $\MT(d-p_{J}-|J|)$ is the spectrum defined in \cite{GMTW 09} and is homotopy equivalent to the Thom-spectrum associated to the formal inverse of the universal vector bundle 
$$U_{d-p_{J}-|J|} \longrightarrow BO(d-p_{J}-|J|).$$ 
We must first fix some notation.
\begin{Notation} The following are some national conventions that will be used in this section and the next. 
\begin{itemize} \itemsep.3cm
\item
For each manifold $P_{i}$ in the sequence $\Sigma$, let 
$N_{P_{i}} \longrightarrow P_{i}$
denote the normal bundle associated to the embedding
$\phi_{i}: P_{i} \hookrightarrow \R^{p_{i} + m_{i}}$ from Section \ref{eq: P-embeddings}. 
\item
For each subset $I \subseteq \kset{k}$, we let 
$N_{P^{I}} \longrightarrow P^{I}$
denote the normal bundle associated to the product embedding
$\phi_{I}: P^{I} \hookrightarrow
\R^{\bar{p}+\bar{m}}_{I}$. 
\end{itemize}
\end{Notation}
For each subset $I \subseteq \langle k \rangle$, there is a factorization
$N_{P^{I}} \; = \; \prod_{i \in I}N_{P_{i}}.$
We fix once and for all tubular neighborhood embeddings,
\begin{equation} \label{eq: tube nbh embeddings P 1}
e_{P_{i}}: N_{P_{i}} \; \hookrightarrow \; \R^{p_{i} + m_{i}} := \R^{\bar{p} + \bar{m}}_{\{i \}}.
\end{equation}
For each subset $J \subseteq \kset{k}$, the product embeddings,
$e_{P^{J}} := \prod_{j \in J}e_{P_{j}}: N_{P^{J}} \; \hookrightarrow \; \R^{\bar{p} + \bar{m}}_{J}$
give tubular neighborhood embeddings for the normal bundles $N_{P^{J}} \longrightarrow P^{J}$. These tubular neighborhoods induce \textit{collapsing maps} 
\begin{equation} \label{eq: collapse map P 1}
c_{P^{J}}: (\R^{\bar{p} + \bar{m}}_{J})^{c} \; \longrightarrow \; \Th(N_{P^{J}}).
\end{equation} 
We will be using these maps throughout the rest of this paper.

\begin{defn} For $n \in \N$ and a subset $J \subseteq \langle k \rangle$, 
let $\Grass{J}{d}{n}$ denote the Grassmannian manifold consisting of all $(d - |J| - p_{J})$- dimensional vector subspaces 
of 
$$\R\times \R^{k}_{J^{c}}\times\R^{\widehat{n}-1+d}\times\R^{\bar{p}+\bar{m}}_{J^{c}}.$$
If $d - |J| - p_{I} \leq 0$, then $\Grass{J}{d}{n}$ is defined to be a single point. 
\end{defn}

For each subset $J \subseteq \langle k \rangle$, the space $\Grass{J}{d}{n}$ is equipped with the \textit{canonical vector bundle}, which we denote by 
$U_{\Sigma_{k},d, n, J} \longrightarrow \Grass{J}{d}{n}.$
We then denote by 
$$U^{\perp}_{\Sigma_{k}, d, n, J} \longrightarrow \Grass{J}{d}{n},$$
the orthogonal compliment bundle.
Let 
$$i_{J, n}: \Grass{J}{d}{n} \; \longrightarrow \; \Grass{J}{d}{n+1}$$
be the embedding induced by the inclusion map
$$\R\times \R^{k}_{J^{c} }\times\R^{d-1+\widehat{n}}\times\R^{\bar{p}+\bar{m}}_{J^{c}} \; \hookrightarrow \; \R\times \R^{k}_{J^{c}}\times\R^{d-1+\widehat{n}+1}\times\R^{\bar{p}+\bar{m}}_{J^{c}}.$$
This embedding yields the bundle map 
$$\xymatrix{
 \epsilon^{1}\oplus U^{\perp}_{\Sigma_{k}, d, n, J} \ar[rr]^{i_{J, n}^{*}} \ar[d] && U^{\perp}_{\Sigma_{k}, d, n+1, J} \ar[d]\\
\Grass{J}{d}{n} \ar[rr]^{i_{J, n}} && \Grass{J}{d}{n+1},}$$
which in-turn induces a map of \textit{Thom-spaces},
\begin{equation} \label{eq: structure map 1} 
\Th(i^{*}_{J, n}):  S^{1}\wedge\Th(U^{\perp}_{\Sigma_{k}, d, n, J}) \longrightarrow \Th(U^{\perp}_{\Sigma_{k}, d, n+1, J}). 
\end{equation}

\begin{defn} For any subset $J \subseteq \kset{k}$, $\MT_{\Sigma_{k}}(d)_{J}$ is the spectrum defined by setting,
$$(\MT_{J}(d)_{\Sigma_{k}})_{(d+n+k)} \; := \; \Th(U^{\perp}_{\Sigma_{k}, d, n, J})\wedge (\R^{\bar{p} + \bar{m}}_{J})^{c}.$$
The structure maps 
$$\sigma^{n}_{J}: (\MT_{\Sigma_{k}}(d)_{J})_{(d+n+k)} \; \longrightarrow \; (\MT_{\Sigma_{k}}(d)_{J})_{(d+n+k+1)}$$
are given by smashing $\Th(i^{*}_{J, n})$ from (\ref{eq: structure map 1}) with the identity on $(\R^{\bar{p} + \bar{m}}_{J})^{c}$. 
\end{defn}

We now assemble the spectra $\MT_{\Sigma_{k}}(d)_{J}$ together into the $k$-cubic spectrum $\MT_{\Sigma_{k}}(d)_{\bullet}$. 
For subsets $I \subseteq J \subseteq \langle k \rangle$, we construct maps
$\kappa_{J, I}: \MT_{\Sigma_{k}}(d)_{J} \; \longrightarrow \; \MT_{\Sigma_{k}}(d)_{I}$
as follows. 

For each subset $J \subseteq \kset{k}$, we denote by $G(\bar{p}, \bar{m}, J)$ the Grassmannian manifold consisting of $p_{J}$-dimensional vector subspaces of $\R^{\bar{p} + \bar{m}}_{J}$. 
We denote by 
$$
U_{\bar{p},\bar{m}, J} \longrightarrow G(\bar{p}, \bar{m}, J) \quad \text{and} \quad
U^{\perp}_{\bar{p},\bar{m}, J} \longrightarrow G(\bar{p}, \bar{m}, J),
$$
the canonical vector bundle and its orthogonal compliment. 
For each pair of subsets $I \subseteq J \subseteq \langle k \rangle$ and $n \geq 0,$ there are maps
\begin{equation} \label{eq: mult then include 1} 
\tau_{J, I}: \Grass{J}{d}{n}\times G(\bar{p}, \bar{m}, J\setminus I) \longrightarrow \Grass{I}{d}{n} 
\end{equation}
given by sending a pair of vector subspaces,  
$$W \subset \; \R^{k}_{J^{c}}\times\R^{d-1+\widehat{n}}\times\R^{\bar{p}+\bar{m}}_{J^{c}},  \quad \quad V \subset \; \R^{\bar{p} + \bar{m}}_{J \setminus I },$$ 
to the subspace of $\R^{k}_{I^{c}}\times\R^{d-1+\widehat{n}}\times\R^{\bar{p}+\bar{m}}_{\langle I^{c} \rangle}$  given by the product,
$\R^{k}_{J\setminus I}\times W\times V$.

The maps $\tau_{J,I}$ are covered by bundle maps
\begin{equation} \label{eq: mult map plus include}
\xymatrix{
U^{\perp}_{\Sigma_{k}, d, n, J}\times U^{\perp}_{\bar{p}, \bar{m}, J\setminus I} \ar[rrr]^{\tau^{*}_{J,I}} \ar[d] &&&  U^{\perp}_{\Sigma_{k}, d, n, I} \ar[d] \\
\Grass{J}{d}{n}\times G(\bar{p}, \bar{m}, J\setminus I) \ar[rrr]^{\tau_{J, I}} &&& \Grass{I}{d}{n},
}
\end{equation} 
which are isomorphisms on the fibres. 
For each subset $J \subseteq \langle k \rangle$, the normal bundle $N_{P^{J}} \rightarrow P^{J}$ admits a Gauss map,
\begin{equation} \label{eq: Gauss map P 1}
\xymatrix{
N_{P^{J}} \ar[d] \ar[rr]^{\gamma^{*}_{P^{J}}} && U^{\perp}_{\bar{p}, \bar{m}, J} \ar[d]\\
P^{J} \ar[rr]^{\gamma_{P^{J}}} && G(\bar{p}, \bar{m}, J),}
\end{equation}
which induces a map of Thom-spaces,
\begin{equation} \label{eq: Thom gamma map} \Th(\gamma^{*}_{P^{J}}): \Th(N_{P^{J}}) \; \longrightarrow \; \Th(U^{\perp}_{\bar{p}, \bar{m}, J}). \end{equation}
For a pair of subsets $I \subseteq J \subseteq \langle k \rangle$, we define 
\begin{equation} \label{eq: k spectra n+d maps} \kappa^{n+d}_{J, I}: (\MT_{\Sigma_{k}}(d)_{J})_{n+d} \; \longrightarrow \; (\MT_{\Sigma_{k}}(d)_{I})_{n+d} \end{equation}
by the composition
$$\xymatrix{
\Th(U^{\perp}_{\Sigma_{k},d, n, J})\wedge (\R^{\bar{p} + \bar{m}}_{ J\setminus I })^{c}\wedge (\R^{\bar{p} + \bar{m}}_{ I })^{c} \ar[rr]^{(1)} && \Th(U^{\perp}_{\Sigma_{k},d, n, J})\wedge\Th(N_{P^{J\setminus I }})\wedge(\R^{\bar{p} + \bar{m}}_{ I })^{c} \ar[dd]^{(2)} \\
\\
 \Th(U^{\perp}_{\Sigma_{k},d, n, J})\wedge\Th(U^{\perp}_{\bar{p}, \bar{m}, J\setminus I})\wedge(\R^{\bar{p} + \bar{m}}_{ I })^{c} && \ar[ll]^{(3)} \Th(U^{\perp}_{\Sigma_{k},d, n, I})\wedge (\R^{\bar{p} + \bar{m}}_{I })^{c}.}
$$ The first map is given is given by $Id\wedge (c_{P^{J\setminus
    I}})\wedge Id$ where $c_{P^{J\setminus I}}$ the collapsing map
from (\ref{eq: collapse map P 1}). The second map is given by
$Id\wedge \Th(\gamma^{*}_{P^{J\setminus I}})\wedge Id$ where
$\Th(\gamma^{*}_{P^{J\setminus I}})$ is the map from (\ref{eq: Thom
  gamma map}). The third map is given by $\Th(\tau^{*}_{J,I})\wedge
Id$ where $\tau^{*}_{J,I}$ is the map from (\ref{eq: mult map plus
  include}).  
 It is easy to check that the maps $\kappa^{n+d}_{J, I}$ are compatible. 
 Namely, for subsets $K \subseteq I \subset J \subseteq \kset{k}$, we have
 $$\kappa^{n+d}_{I, K}  \circ \kappa^{n+d}_{J, I}  \; = \; \kappa^{n+d}_{J, K}.$$
From this compatibility, the maps $\kappa^{n+d}_{J, I}$ make $(\MT_{\Sigma_{k}}(d)_{\bullet})_{d+n}$ into a $k$-cubic space for each integer $n \geq 0$. 
 It follows from the construction that for each pair of subsets $I \subseteq J \subseteq \langle k \rangle$, the diagram
 $$\xymatrix{
 (\MT_{J}(d)_{\Sigma_{k}})_{d +n+1} \ar[rr]^{\kappa^{n+1}_{J,I}} && (\MT_{I}(d)_{\Sigma_{k}})_{d+n+1}\\
 S^{1}\wedge(\MT_{J}(d)_{\Sigma_{k}})_{n+d} \ar[u]^{\sigma^{n}} \ar[rr]^{\kappa^{n}_{J, I}} && S^{1}\wedge(\MT_{}(d)_{\Sigma_{k}})_{n+d} \ar[u]^{\sigma^{n}}}$$
 commutes, thus the $\kappa^{n}_{J,I}$ piece together to define maps of spectra
 \begin{equation} \label{eq: k-spectra maps} 
 \kappa_{J,I}: \MT_{J}(d)_{\Sigma_{k}} \; \longrightarrow \; \MT_{I}(d)_{\Sigma_{k}}
  \end{equation}
 such that $\kappa_{I, K}  \circ \kappa_{J, I}  \; = \; \kappa_{J, K}$ for any $K \subseteq I \subseteq J \subseteq \langle k \rangle$. 
We obtain a $k$-cubic spectrum, $\MT_{\Sigma_{k}}(d)_{\bullet}$. 
 We then denote
\begin{equation} \label{eq: total cofibre notation}  
\xymatrix{
\MT_{\Sigma_{k}}(d) \; := \; \tCofibre_{\kset{k}}\MT_{\Sigma_{k}}(d)_{\bullet}. 
}
\end{equation}
Recall that by definition, 
\begin{equation} \label{eq: total cofibre spectrum 1}
\xymatrix{
\MT_{\Sigma_{k}}(d) \; = \;
\Cofibre\bigg(\tCofibre_{\kset{k-1}}\MT_{\Sigma_{k}}(d)_{\bullet, k}
  \rightarrow \tCofibre_{\kset{k-1}}\MT_{\Sigma_{k}}(d)_{\bullet,
    \bar{k}}\bigg).
    }
 \end{equation}   
By inspection one sees that
$$\MT_{\Sigma_{k}}(d)_{\bullet, k} \; = \; \Sigma^{-1}\MT_{\Sigma_{k-1}}(d-p_{k}-1)_{\bullet}, \quad \quad \MT_{\Sigma_{k}}(d)_{\bullet, \bar{k}} \; = \; \MT_{\Sigma_{k-1}}(d)_{\bullet}.$$
Combining this observation with (\ref{eq: total cofibre spectrum 1}), we obtain a cofibre sequence
$$\Sigma^{-1}\MT(d-p_{k}-1)_{\Sigma_{k-1}} \; \longrightarrow \; \MT(d)_{\Sigma_{k-1}} \; \longrightarrow \; \MT(d)_{\Sigma_{k}}.$$
Continuing this cofibre sequence one term to the right and then applying the functor $\Omega^{\infty}( \underline{\hspace{.2cm}})$ yields a homotopy-fibre sequence,
$$\xymatrix{
\Omega^{\infty}\MT_{\Sigma_{k-1}}(d) \; \longrightarrow \; \Omega^{\infty}\MT_{\Sigma_{k}}(d) \; \longrightarrow \; \Omega^{\infty}\MT_{\Sigma_{k-1}}(d-p_{k}-1).
}$$
For the case that $k = 1$ and so $\Sigma_{k} = (P)$ for some closed manifold $P$, the above fibre sequence is the fibre sequence studied in \cite[Page 38]{P 13}. 

\section{The Main Theorem} \label{section: Main Theorem}
In this section we identify the homotopy type of the space $|\mathbf{D}^{\Sigma_{k}}_{d}|$. 
We will prove:
\begin{theorem} \label{thm: Main Theorem} There is a weak homotopy equivalence
$|\mathbf{D}^{\Sigma_{k}}_{d}| \simeq \Omega^{\infty-1}\MT(d)_{\Sigma_{k}}.$
\end{theorem}
Combining this with the weak homotopy equivalence $|\mathbf{D}^{\Sigma_{k}}_{d}| \; \simeq \; B\mathbf{Cob}_{d}^{\Sigma_{k}}$ proved in Section \ref{subsection: the cobordism category sheaf} yields the weak homotopy equivalence 
$B\mathbf{Cob}_{d}^{\Sigma_{k}} \; \simeq \; \Omega^{\infty -1}\MT(d)_{\Sigma_{k}}$ asserted in Theorem \ref{thm: weak homotopy equivalence}. 
By Theorem \ref{thm: loopspace eq 1} it will suffice to prove the homotopy equivalence
$|\mathbf{D}^{\Sigma_{k}}_{d}| \; \simeq \; \Omega^{\infty -1}_{\kset{k}}\MT_{\Sigma_{k}}(d)_{\bullet}.$

\subsection{A parametrized Pontryagin Thom construction} \label{subsection: A Parametrized Thom Pontryagin Construction}
We now work to construct a zig-zag of weak homotopy equivalences
$\xymatrix{
|\mathbf{D}^{\Sigma_{k}}_{d}| \; \ar@{<~>}[r] & \; \Omega^{\infty -1}_{\kset{k}}\MT_{\Sigma_{k}}(d)_{\bullet}.
}$
For each $n \in \N$, we will need a modified version of the sheaf $\mathbf{D}^{\Sigma_{k}}_{d}$.
We need a preliminary technical definition. 
\begin{defn} 
Let $p: Y \longrightarrow X$ be a submersion. Let $i_{C}: C
\hookrightarrow Y$ be a smooth submanifold and suppose that
$p\mid_{C}$ is still a submersion. A \textit{vertical} tubular
neighborhood for $C$ in $Y$ consists of a smooth vector bundle $q: N
\rightarrow C$ (which in our case will always be the normal bundle of
$C$) with zero section $s$, along with an open embedding $e: N
\longrightarrow Y$ such that $e\circ s = i_{C}$, and $p\circ e =
p\circ i_{C} \circ q$. \end{defn}
\begin{defn} \label{defn: D vert sheaf} For $X \in \mathcal{X}$ we define $\mathbf{D}^{\Sigma_{k}, \nu}_{d, n}(X)$ to be the set pairs $(W, e)$ such that
\begin{enumerate} \itemsep.2cm
\item[i.] $W \in \mathbf{D}^{\Sigma_{k}}_{d, n}(X)$,
\item[ii.] The map 
$$e: N_{W} \hookrightarrow X\times\R\times\R^{k}_{+}\times\R^{d-1+\widehat{n}}\times\R^{\bar{p}+\bar{m}},$$
  where $N_{W}$ is the normal bundle of $W$, is a \textit{vertical
    tubular neighborhood} of $W$ with respect to the submersion
  $\pi: W \longrightarrow X$, that satisfies the following condition:
   for each subset $I \subseteq \langle k\rangle$,
  the restriction $e\mid_{\partial_{I}W}$ has the factorization,
$e\mid_{N_{\partial_{I}W}} \; = \; e_{N_{\beta_{I}W}}\times e_{P^{I}}$
where 
$$e_{N_{\beta_{I}W}}: N_{\beta_{I}W} \hookrightarrow X\times\R \times\R^{k}_{+, I^{c}}\times\R^{d-1+\widehat{n}}\times\R^{\bar{p}+\bar{m}}_{I^{c}},
$$
is a vertical tubular neighborhood for $\beta_{I}W$, $N_{\beta_{I}W}$ the is normal bundle, and 
$$e_{P^{I}}: N_{P^{I}} \hookrightarrow \R^{\bar{p}+\bar{m}}_{I}$$
is the tubular neighborhood for $N_{P^{I}}$ specified in (\ref{eq: tube nbh embeddings P 1}). 
\end{enumerate}
\end{defn}
Clearly, $\mathbf{D}^{\Sigma_{k}, \nu}_{d, n}$ satisfies the sheaf condition. 
For each $n$ there is a forgetful map
\begin{equation} \label{equation: forgetful map}
\mathbf{D}^{\Sigma_{k}, \nu}_{d, n} \longrightarrow \mathbf{D}^{\Sigma_{k}}_{d}.
\end{equation}
It follows easily from the existence of \textit{tubular neighborhoods} that the induced map in the limit
$$\colim_{n\to\infty}|\mathbf{D}^{\Sigma_{k}, \nu}_{d, n}| \stackrel{\simeq} \longrightarrow |\mathbf{D}^{\Sigma_{k}}_{d}|$$
is a homotopy equivalence. 
For $X \in \Ob(\mathcal{X})$, we define 
$$\mathcal{Z}^{\Sigma_{k}}_{d}(X) := \Maps(X\times\R, \; \Omega^{\infty-1}_{\langle k \rangle}\MT_{\Sigma_{k}}(d)_{\bullet}).$$
The assignment $X \mapsto \mathcal{Z}^{\Sigma_{k}}_{d}(X)$ defines a sheaf on $\mathcal{X}$. 
The representing space $|\mathcal{Z}^{\Sigma_{k}}_{d}|$, is weakly homotopy equivalent to $\Omega^{\infty-1}_{\langle k \rangle}\MT_{\Sigma_{k}}(d)_{\bullet}$. 
We now proceed for each $n$ to define a natural transformation
$T^{\Sigma_{k}}_{d,n}: \mathbf{D}^{\Sigma_{k}, \nu}_{d, n} \longrightarrow \mathcal{Z}^{\Sigma_{k}}_{d},$ based on the \textit{Pontryagin-Thom construction}. 
\begin{Construction}
Let $X \in \Ob(\mathcal{X})$ and let $(W, e) \in \mathbf{D}^{\Sigma_{k}, \nu}_{d, n}(X)$. 
\begin{enumerate} \itemsep.4cm
\item[(a)]
For each subset $J \subseteq \kset{k}$, let 
$N_{\partial_{J}W} \; \longrightarrow \; \partial_{J}W$
denote the normal bundle associated to $\partial_{J}W$ as a submanifold of $\R\times\R^{k}_{+, J^{c}}\times\R^{d-1+\widehat{n}}\times\R^{\bar{p} + \bar{m}}$. 
For each subset $I \subseteq J$, we have
$N_{\partial_{J}W}\mid_{\partial_{I}W} \; = \; N_{\partial_{I}W}$
(see Remark \ref{remark: normal bundles}).
Furthermore, for each for each subset $J \subseteq \langle k \rangle$ there is a factorization 
$$N_{\partial_{J}}W = N_{\beta_{J}W}\times N_{P^{J}}$$ 
where $N_{\beta_{J}W} \rightarrow \beta_{J}W$ is the normal  bundle for $\beta_{J}W \subset \R\times\R^{k}_{+, J^{c}}\times\R^{d-1+\widehat{n}}\times\R^{\bar{p} + \bar{m}}_{J^{c}}$ and $N_{P^{J}} \rightarrow P^{J}$ is the normal bundle for $\phi_{J}(P^{J}) \subset \R^{\bar{p} + \bar{m}}_{J}$. 
\item[(b)]
For each $J \subseteq \kset{k}$, the \textit{Gauss map} for the normal bundle $N_{\partial_{J}W} \rightarrow \partial_{J}W$ factors as the product,
$$\xymatrix{
N_{\partial_{J}W} \; = \; N_{\beta_{J}W}\times N_{P^{J}} \ar[rrr]^{\gamma^{*}_{\beta_{J}}\times \gamma^{*}_{P^{J}}} \ar[d] &&& U^{\perp}_{\Sigma_{k}, d, n, J}\times U^{\perp}_{\bar{p}, \bar{m}, J} \ar[d] \\
\partial_{J}W \; = \; \beta_{J}W\times P^{J} \ar[rrr]^{\gamma_{\beta_{J}}\times \gamma_{P^{J}}} &&&  \; \Grass{J}{d}{n}\times G(\bar{p}, \bar{m}, J).}$$
This bundle map $\gamma^{*}_{\beta_{J}}\times \gamma^{*}_{P^{J}}$ induce maps of Thom-spaces,
\begin{equation} \label{eq: Thom space gauss maps}
 \widehat{\gamma_{\beta_{J}W}}\wedge \widehat{\gamma}_{P^{J}}: \Th(N_{\beta_{J}W})\wedge \Th(N_{P^{J}}) \; \longrightarrow \; \Th(U^{\perp}_{\Sigma_{k}, d, n, J})\wedge \Th(U^{\perp}_{\bar{p}, \bar{m}, J}).
 \end{equation}
 \item[(c)]
The vertical tubular neighborhood 
$$
e_{\beta_{J}W}: N_{\beta_{J}W} \longrightarrow X\times\R\times\R^{k}_{+, J^{c}}\times\R^{d-1+\widehat{n}}\times\R^{\bar{p} + \bar{m}}_{J^{c}}$$
induces a collapsing map 
$$
c_{\beta_{J}W}: X\times\R\times(\R^{k}_{+, J^{c}}\times\R^{d-1+\widehat{n}}\times\R^{\bar{p} + \bar{m}}_{J^{c}})^{c} \longrightarrow \Th(N_{\beta_{J}W}).
$$
Using the collapse map $c_{\beta_{J}W}$ together with $\widehat{\gamma}_{\beta_{J}W}$ from (\ref{eq: Thom space gauss maps}), we define
$$
\widehat{T}^{\Sigma_{k}}_{d+n}(W)_{J}:  
(X\times\R)\times (\R^{k}_{+, J^{c}}\times\R^{d-1+\widehat{n}}\times\R^{\bar{p} + \bar{m}}_{J^{c}})^{c}\wedge(\R^{p_{k}+m_{k}}_{J})^{c} \; \longrightarrow \; \Th(U^{\perp}_{\Sigma_{k}, d, n J})\wedge (\R^{\bar{p}+\bar{m}}_{J})^{c} 
$$
by the formula,
$$\bigg((x, t), \; y, \; z\bigg) \; \mapsto \; \bigg(\widehat{\gamma}_{\beta_{J}W}(c_{\beta_{J}W}((x, t), \; y)) ,\; \; z\bigg).$$
\item[(d)]
By our construction it follows that for each pair of subsets $I \subseteq J \subset \kset{k}$, the diagram 
\begin{equation} \label{eq: commuting I J 1}
\xymatrix{
X\times\R\times(\R^{k}_{+, I^{c}}\times\R^{d-1+\widehat{n}}\times\R^{\bar{p} + \bar{m}})^{c} \ar[rrr]^{\widehat{T}^{\Sigma_{k}}_{d+n}(W)_{I}} &&& \Th(U^{\perp}_{\Sigma_{k}, d, n, I})\wedge (\R^{\bar{p}+\bar{m}}_{I})^{c}\\
X\times\R\times(\R^{k}_{+, J^{c}}\times\R^{d-1+\widehat{n}}\times\R^{\bar{p} + \bar{m}})^{c} \ar[rrr]^{\widehat{T}^{\Sigma_{k}}_{d+n}(W)_{J}} \ar[u] &&& \Th(U^{\perp}_{\Sigma_{k}, d, n J})\wedge (\R^{\bar{p}+\bar{m}}_{J})^{c} \ar[u]^{\kappa_{J,I}^{n+d}}}
\end{equation}
commutes, where the left vertical map is the inclusion and the maps $\kappa_{J,I}^{n+d}$ are the \textit{edges} in the $k$-cubic space $(\MT_{\Sigma_{k}}(d)_{\bullet})_{n+d}$.
It follows that the maps $\widehat{T}_{d+n}^{\Sigma_{k}}(W)_{J}$, for $J \subseteq \langle k \rangle$, are compatible and thus yield a morphism of 
 $k$-cubic spaces
$$\widehat{T}^{\Sigma_{k}}_{d+n}(W)_{\bullet}: \;  X\times\R\times D^{\kset{k}, d-1 +n}_{\bullet} \; \longrightarrow \; (\MT_{\Sigma_{k}}(d)_{\bullet})_{n+d}.$$
\item[(e)] Let $Ad_{d, n}$ denote the adjoint homeomorphism,
$$ \xymatrix{
\Maps_{\kset{k}}\bigg(X\times\R\times D^{\kset{k}, d-1 +n}_{\bullet}, \; (\MT_{\Sigma_{k}}(d)_{\bullet})_{n+d}\bigg) \ar[r]^{\ \ \ \cong} & \Maps\bigg(X\times\R, \; \Omega^{d-1+n}_{\kset{k}}(\MT_{\Sigma_{k}}(d)_{\bullet})_{n+d}\bigg).}
$$
Finally, we define the natural transformation 
$T^{\Sigma_{k}}_{d,n}: \mathbf{D}^{\Sigma_{k}, \nu}_{d, n} \longrightarrow \mathcal{Z}^{\Sigma_{k}}_{d},$
by setting 
$$T_{d+n}^{\Sigma_{k}}(W) := \text{Ad}_{n, d}\bigg(\widehat{T}^{\Sigma_{k}}_{d+n}(W)_{\bullet}\bigg).$$
\end{enumerate}
\end{Construction}

\begin{proof}[Proof of Theorem \ref{thm: Main Theorem}] 
We prove the result by induction on $k$. 
In the limit $n \to \infty$, the natural transformations $T^{\Sigma_{k}}_{d, n}$ induce a map of homotopy fibre-sequences,
\begin{equation} \label{eq: fibre map limit 2}
\xymatrix{
|\mathbf{D}^{\Sigma_{k-1}}_{d}| \ar[rr] && |\mathbf{D}^{\Sigma_{k}}_{d}| \ar[rr] && |\mathbf{D}^{\Sigma_{k-1}}_{d-p_{k}-1}| \\
{\displaystyle \colim_{n\to\infty}}|\mathbf{D}^{\Sigma_{k-1}, \nu}_{d, n}| \ar[u]^{\simeq} \ar[rr] \ar[d] && {\displaystyle \colim_{n\to\infty}}|\mathbf{D}^{\Sigma_{k}, \nu}_{d, n}| \ar[u]^{\simeq} \ar[d] \ar[rr] && {\displaystyle \colim_{n\to\infty}}|\mathbf{D}^{\Sigma_{k-1}, \nu}_{d-p_{k}-1, n}| \ar[u]^{\simeq} \ar[d]\\
|\mathcal{Z}^{\Sigma_{k-1}}_{d}| \ar[d]^{\simeq} \ar[rr] && |\mathcal{Z}^{\Sigma_{k}}_{d}| \ar[d]^{\simeq} \ar[rr] && |\mathcal{Z}^{\Sigma_{k-1}}_{d-p_{k}-1}| \ar[d]^{\simeq} \\
\Omega^{\infty-1}_{\kset{k-1}}\MT_{\Sigma_{k-1}}(d)_{\bullet} \ar[rr] && \Omega^{\infty-1}_{\kset{k}}\MT_{\Sigma_{k}}(d)_{\bullet} \ar[rr] && \Omega^{\infty-1}_{\kset{k-1}}\MT_{\Sigma_{k-1}}(d-p_{k}-1)_{\bullet} 
}
\end{equation}
In the case $k = 1$, the map
$$|T^{\Sigma_{1}}_{d}|: {\displaystyle \colim_{n\to\infty}}|\mathbf{D}^{\Sigma_{1}, \nu}_{d, n}| \longrightarrow |\mathcal{Z}^{\Sigma_{1}}_{d}|$$
is the map from \cite[Theorem 7.1]{P 13}, and thus is a homotopy equivalence for all $d \geq 0$; this proves the base case of the induction. 
Now assume that 
$$|T^{\Sigma_{k-1}}_{d}|: {\displaystyle \colim_{n\to\infty}}|\mathbf{D}^{\Sigma_{k-1}, \nu}_{d, n}| \longrightarrow |\mathcal{Z}^{\Sigma_{k-1}}_{d}|$$
is a homotopy equivalence for all $d$. 
Then, the first and third columns of (\ref{eq: fibre map limit 2}) are homotopy equivalences. The result follows from an application of the \textit{five lemma} on the map of long-exact sequences in homotopy groups induced by the fibre-sequences from (\ref{eq: fibre map limit 2}). 
\end{proof}

\appendix
\section{} \label{embedding restriction}
The following lemma is based on the proof of the main result from \cite{L 63}. 
This lemma implies Lemma \ref{lemma: serre fibration embeddings} from Section \ref{section: mapping spaces}. 
\begin{lemma} \label{lemma: serre fibration embeddings} 
Let $M$ be a closed $\Sig{k}{0}$-manifold. 
Let $0 \leq \ell < k$ be integers. 
Then for any $n \in \N$, the restriction map 
 $$r_{\ell+1}: \mathcal{E}_{\Sig{k}{\ell}, n}(M) \; \longrightarrow \; \mathcal{E}_{\Sig{k-1}{\ell}, n}(\partial_{\ell}M), \quad \varphi \mapsto \varphi|_{\partial_{\ell+1}M}.$$ 
 is a locally trivial fibre-bundle.
 \end{lemma}
 \begin{proof}
 We will prove the theorem explicitly in the case that $\ell = 0$. 
 Then general case is proven in exactly the same way. 
 
 Let $f \in \mathcal{E}_{\Sig{k-1}{0}, n}(\partial_{1}M, X)$. 
 Notice, that the space $\R^{k}_{+}\times \R^{n}$ is a (non-compact) $\Sigma^{0}_{k}$-manifold with $\partial_{\ell}(\R^{k}_{+}\times \R^{n}) = \R^{k}_{+,\{\ell\}^{c}}\times \R^{n}$ for all $\ell \in \langle k \rangle$. 
 We may then form the mapping spaces $C^{\infty}_{\Sigma^{0}_{k}}(\R^{k}_{+}\times \R^{n})$ and $\Diff_{\Sigma^{0}_{k}}(\R^{k}_{+}\times \R^{n})$ as in Definition \ref{defn: P-morphism}. 
 We then let
 $C^{\infty}_{\Sigma^{0}_{k}}(\R^{k}_{+}\times \R^{n})^{c}$ denote the subspace of $C^{\infty}_{\Sigma^{0}_{k}}(\R^{k}_{+}\times \R^{n})$ consisting of all smooth $\Sigma^{0}_{k}$-maps $f$, such that $f$ is the identity outside of some compact subset. 
 We define $\Diff_{\Sigma^{0}_{k}}(\R^{k}_{+}\times \R^{n})^{c}$ similarly.
To prove the lemma, it will suffice to find a neighborhood $U \subset \mathcal{E}_{\Sig{k-1}{0}, n}(\partial_{1}M, X)$ of $f$ and a continuous map 
$\eta: U \longrightarrow \Diff_{\Sigma^{0}_{k}}(\R^{k}_{+}\times \R^{n})^{c}$
that satisfies:
\begin{enumerate} \itemsep.2cm
\item[(a)] $\eta(f) = Id_{\R^{k}_{+}\times\R^{n}}$,
\item[(b)] $r_{1}(\eta(g)\circ \hat{f}) = g$ for all $g \in U$ and $\hat{f} \in r_{1}^{-1}(f).$
\end{enumerate}
With such a map $\eta$ constructed, the map
$$U\times r_{1}^{-1}(f) \longrightarrow r^{-1}_{1}(U), \quad (g, \hat{f}) \mapsto \eta(g)\circ\hat{f}$$
defines a trivialization of $r_{1}$ over the neighborhood $U$. 
Thus, the proof of the lemma will be complete.

For any real number $\delta > 0$, let $N_{\delta}$ be the tubular neighborhood about $f(\partial_{1}W)$ in $\R^{k-1}_{+}\times\R^{n}$ and let 
$$\pi: N_{\delta} \longrightarrow f(\partial_{1}W)$$ 
be the projection map. 
Let $U$ be an open neighborhood of $f$ in $\mathcal{E}_{\Sig{k-1}{0}, n}(\partial_{1}M, X)$ with the property that $g(\partial_{1}W) \subset N_{\delta/2}$ for all $g \in U$. 

Let $\lambda: [0, \infty) \longrightarrow [0,1]$ be a non-increasing smooth function that satisfies:
\begin{itemize} \itemsep.2cm
\item $\lambda(t) = 1$ for $t \leq \frac{\delta}{2}$,  
\item $\lambda(t) = 0$ for $t \geq \frac{3\delta}{4}$. 
\end{itemize}
The let $\rho: [0,1] \longrightarrow [0,1]$ be a non-decreasing, smooth function that satisfies: 
 \begin{itemize} \itemsep.2cm
 \item $\rho(t) = 1$ for $t \leq \frac{1}{4}$, 
 \item $\rho(t) = 0$ for $t \geq \frac{3}{4}$. 
 \end{itemize}
We define a map
$$\xymatrix{
\eta: U\times[0, 1] \longrightarrow   C^{\infty}_{\Sigma^{0}_{k-1}}(\R^{k-1}_{+}\times\R^{n})^{c}
}
$$
by the formula
\begin{equation} \label{equation: main function 1}
\eta(g, t)(\bar{y}) \; = \; \bar{y} + \rho(t)\cdot\lambda\bigg(\text{dist}(\bar{y}, \; f(\partial_{1}W))\bigg)\cdot \bigg(g(f^{-1}(\pi(\bar{y}))) - \pi(\bar{y})\bigg),
\end{equation}
where $g \in U$, $t \in [0,1]$ and $\bar{y} \in \R^{k-1}_{+}\times\R^{n}$. 
In the above formula,
$\text{dist}(\bar{y}, \; f(\partial_{1}W))$ is the Euclidean distance and all sums are taken coordinatewise. 
(This function is similar to the function constructed in the proof of the main theorem from \cite{L 63}.)
The following properties of (\ref{equation: main function 1}) can be verified directly:
\begin{enumerate} \itemsep.3cm
 \item[i.] $\eta(f, t)(\bar{y}) = \bar{y}$ \quad for all $t \in [0,1]$ \quad and \quad $\bar{y} \in \R^{k}_{+}\times\R^{n}$, 
\item[ii.] $\eta(g, 0)(f(x)) = g(x)$ \quad for all \quad $x \in \partial_{1}W$.
\item[iii.] $\eta(g, 1)(\bar{y}) = \bar{y}$ \quad for all $g \in \tilde{U}$ and $\bar{y} \in \R^{k}_{+}\times\R^{n}$.
\end{enumerate}
Let 
$$\xymatrix{
\widehat{\eta}: U \longrightarrow \Maps\bigg([0,1], \; C^{\infty}_{\Sigma^{0}_{k-1}}(\R^{k-1}_{+}\times\R^{n})^{c}\bigg)
}$$
be the map determined by $\eta$ under adjunction. 
By property i., $\widehat{\eta}(f)(t) \in C^{\infty}_{\Sigma^{0}_{k-1}}(\R^{k-1}_{+}\times\R^{n})^{c}$ is the identity map for all $t \in [0,1]$. 
Since 
$$\xymatrix{
\Diff_{\Sigma^{0}_{k-1}}(\R^{k-1}_{+}\times\R^{n})^{c} \subset C^{\infty}_{\Sigma^{0}_{k-1}}(\R^{k-1}_{+}\times\R^{n})^{c},
}$$ 
is an open subset (see Proposition \ref{prop: P-diff open set} and \cite[Proposition 3.1]{P 13}),
it follows that 
$$\xymatrix{
\Maps\bigg([0,1], \; \Diff_{\Sigma^{0}_{k-1}}(\R^{k-1}_{+}\times\R^{k})^{c}\bigg) \; \subset \; \Maps\bigg([0,1], \; C^{\infty}_{\Sigma^{0}_{k-1}}(\R^{k-1}_{+}\times\R^{n})^{c}\bigg)
}$$
is an open subset in the compact-open topology. 
Thus, since $\widehat{\eta}(f)(t): \R^{k-1}_{+}\times\R^{n} \longrightarrow \R^{k-1}_{+}\times\R^{n}$ is the identity for all $t \in [0,1]$ (and hence is a diffeomorphism for all $t \in [0,1]$), we may choose a smaller neighborhood $\bar{U} \subset U$ of $f$ such that the restriction 
$\widehat{\eta}|_{\bar{U}}$ has its image contained in the space
$$\xymatrix{
\Maps\bigg([0,1], \; \Diff_{\Sigma^{0}_{k-1}}(\R^{k-1}_{+}\times\R^{k})^{c}\bigg).
}$$
We then define 
\begin{equation}
\xymatrix{
\bar{\eta}: \bar{U} \longrightarrow \Diff_{\Sigma^{0}_{k}}(\R_{+}\times(\R^{k-1}_{+}\times\R^{n}))^{c}
}
\end{equation}
by the formula
$$\bar{\eta}(g)(t, \bar{y}) = 
\begin{cases} \itemsep.2cm
(t, \; \;  \eta(g, t)(\bar{y})) &\quad \text{for $(t, \bar{y}) \in [0,1]\times(\R^{k-1}_{+}\times\R^{n})$,}\\
(t, \; \; \bar{y}) &\quad \text{for $(t, \bar{y}) \in (1, \infty)\times(\R^{k-1}_{+}\times\R^{n})$.}
\end{cases}
$$
Using properties i., ii., and iii. it follows that $\bar{\eta}$ is well defined and it satisfies the required conditions (a) and (b) given above. 
This completes the proof of the lemma. 
\end{proof}

\end{document}